%% file: relativeadjunction.tex
\documentclass[12pt]{amsart}
\font\emailfont=cmtt10

\usepackage{amsmath,amsthm,amsfonts,amscd,flafter,epsf}
\usepackage{graphicx}
\usepackage{epsfig}
\usepackage{psfrag}
\usepackage{subcaption}
\usepackage{mathtools}
\usepackage[all]{xy}
\usepackage[color=blue!20!white,textsize=tiny]{todonotes}

\input{macros}

\usepackage{enumitem}

\commentable{

\title[{Knot Floer homology and relative adjunction inequalities}] 
{Knot Floer homology and relative adjunction inequalities}

}

\author[Matthew Hedden]{Matthew Hedden}
\address{Department of
Mathematics, Michigan State University \newline
\indent{\emailfont{mhedden@math.msu.edu}}}

\author{Katherine Raoux}
\address{Department of Mathematics, Michigan State University \newline\indent{\emailfont{raouxkat@msu.edu}}}

\thanks{MH gratefully acknowledges support from NSF grants DMS-0706979,  DMS-0906258,  CAREER DMS-1150872,  DMS-1709016, DMS-2104664 and an Alfred P. Sloan Research Fellowship.  KR was partially supported by an AWM Mentoring Travel Grant}

\begin{document}
\maketitle

\begin{abstract} We establish inequalities that constrain the genera of smooth cobordisms between knots in 4-dimensional cobordisms.  These ``relative adjunction inequalities" improve the adjunction inequalities for closed surfaces which have been instrumental in many topological applications of gauge theory.  The relative inequalities refine the latter  by incorporating numerical invariants of knots in the boundary associated to Heegaard Floer homology classes determined by the $4$-manifold.  As a corollary, we produce a host of concordance invariants  for knots in a general 3-manifold, one such invariant for every non-zero Floer class.  We apply our results to produce analogues of the Ozsv{\'a}th-Szab{\'o}-Rasmussen concordance invariant for links, allowing us to reprove the link version of the Milnor conjecture, and, furthermore, to show that knot Floer homology detects strongly quasipositive fibered links.   

\end{abstract}

\section{Introduction}
A complex curve embedded in a complex surface satisfies a so-called ``adjunction formula" that computes its Euler characteristic  in terms of its self-intersection number and  pairing with the first Chern class of the complex structure; see, for example, \cite{GS}. Applied to a smooth algebraic curve $V_d\subset \CP{2}$,  one obtains the classical formula expressing its genus in terms of the degree $d$ of its defining homogenous polynomial:  $2g=(d-1)(d-2)$.   The Thom conjecture asserts that any smoothly embedded surface in the homology class of $V_d$ has genus at least this large.  

The advent of gauge theory brought tools that could tackle this surprising conjecture.  These  take the form of ``adjunction inequalities", which constrain the genera of smoothly embedded surfaces in 4-manifolds possessing non-vanishing gauge-theoretic invariants:

\begin{Adjunction Inequality}[\cite{KMThom,MST,HolDiskFour}] \label{thm:closed}
Let $X$ be a smooth, closed, oriented $4$-manifold satisfying $b_2^+(X)>1$, and let $\spinct$ be a $\SpinC$ structure on $X$ with non-zero Seiberg-Witten or \os\ invariant.  Then $$ |\langle c_1(\spinct),[\Sigma]\rangle | + [\Sigma]^2 \le 2g(\Sigma)-2,$$ 
where $\Sigma \subset X$ is any smoothly embedded oriented surface satisfying $[\Sigma]^2\ge 0$.
\end{Adjunction Inequality} 
\noindent Here, $b_2^+(X)$ denotes the dimension of a maximal subspace on which the intersection form on $H_2(X;\R)$ is positive definite and $g(\Sigma)$ denotes the genus.  A similar inequality, predating the above, holds for manifolds with non-trivial Donaldson polynomial invariant \cite{KM1,KM2}.   These theorems have been generalized in several directions, most notably to the situation where $b_2^+(X)=1$ and the square of $\Sigma$ is arbitrary \cite{SympThom,HigherAdj}.  

It is difficult to overstate the importance of these inequalities in the study of smooth 4-manifolds.  A particular triumph was their affirmation of a general Thom conjecture:
\begin{sympthom} [\cite{SympThom}] Let $\Sigma\subset (X,\omega)$ be a smoothly embedded symplectic surface in a symplectic $4$-manifold.  Then $\Sigma$ minimizes genus amongst all smoothly embedded surfaces in its homology class.
\end{sympthom} 
\noindent In this generality, the theorem was proved by \ons \ \cite{SympThom}, but important special cases were proved by a collection of authors, most notably the case of holomorphic curves in  $\CP{2}$ by Kronheimer and Mrowka \cite{KMThom}.\\

The  purpose of this article is to prove a relative version of the adjunction inequality for properly embedded surfaces in 4-manifolds with boundary. Such a surface  intersects the boundary 3-manifold in a knot or link, and our theorem refines the adjunction inequality with numerical invariants of this link derived from  knot Floer homology.  A knot $K\subset Y$ determines a filtration of the Heegaard Floer homology $\HFa(Y)$.  Given a non-zero Floer class $\beta\in\HFa(Y)$, we thereby obtain a number $\tau_\beta(Y,K)$ that records the filtration level of $\beta$. Roughly, our main theorem says that the number  $\tau_\beta(Y,K)$ bounds the genera of properly embedded surfaces with boundary $K$ in certain $4$-manifolds with boundary. The  Floer-theoretic assumption on these $4$-manifolds which plays the role of a non-vanishing Seiberg-Witten invariant is nontriviality of the map induced by $W$ on the hat Floer homology. Specifically, we must have $F_{W,\spinct}(\alpha)=\beta$, where $\alpha$ and $\beta$ are non-zero Heegaard Floer homology classes for $Y_1$ and $Y_2$, respectively, and $\spinct$ is a $\SpinC$ structure on $W$.

 \begin{thm}\label{thm:relativeadjunction}
 Let $W$ be a smooth, compact,  oriented 4-manifold with $\partial W =-Y_1\sqcup Y_2$.  Let $K_1 \subset Y_1$ and $K_2\subset Y_2$ be  rationally null-homologous knots. If $F_{W,\spinct}(\alpha)=\beta\neq 0$, then 
 \begin{equation}\label{eqn:reladjunction}
 \langle c_1(\spinct),[\Sigma]\rangle+[\Sigma]^2+2(\taub-\taua)\leq 2g(\Sigma)
 \end{equation}
where $\Sigma$ is any oriented surface, smoothly  and properly embedded  with  $\partial \Sigma=-K_1\sqcup K_2$.
\end{thm}
\noindent If  $\Sigma$ is disconnected we interpret its genus to be the sum of the genera of its components. Note that a surface with $\partial \Sigma=-K_1\sqcup K_2$ exists if and only if $[K_1]=[K_2]$ in $ H_1(W;\Z)$. 

The left two terms of Equation \eqref{eqn:reladjunction} warrant some explanation. To define them, we lift $[\Sigma]$ to a class in $H_2(W;\Q)$, and consider the relevant $\Q$-valued pairing and self-intersection number.  The existence of a lift is guaranteed by our hypothesis that the knots are rationally null-homologous, i.e. that $0=[K_i]\in H_1(Y_i;\Q)$.  If the knots are null-homologous,  then we can lift $[\Sigma]$ to a class in $H_2(W;\Z)$ and the terms on the left will be integers.   Note, though, that in both cases the  lift of $[\Sigma]$ is typically not canonical; there is an ambiguity coming from $H_2(\partial W)\cong H_2(Y_1)\oplus H_2(Y_2)$.  This ambiguity is present in the definition of the filtration on Floer homology and  we show that the sum of the terms on the left hand side of Equation \eqref{eqn:reladjunction} is independent of the lift. Details are discussed in Section \ref{sec:mainproof}.

Special cases of Theorem \ref{thm:relativeadjunction} have appeared throughout the literature.  The first, which was the initial inspiration for this work, is due to Rasmussen \cite{RasThesis} and \ons \ \cite{FourBall} and treats the case of knots in the $3$-sphere, $S^3$. In this setting, there is a unique non-zero Floer class, and the corresponding invariant is denoted $\tau(K)$. They prove the relative adjunction inequality for surfaces in negative definite 4-manifolds with boundary $S^3$, and show that $\tau(K)$ is a concordance invariant (indeed a concordance homomorphism).  Since then, genus bounds and concordance invariance have been established in various settings for the  $\tau_\alpha$ invariant corresponding to the subspace of Floer homology arising from the stable image of $U^n$ \cite{GRS,Raoux, HLL}.  Our theorem encompasses all these results.

We  note that the theorem above, in the null-homologous case, can alternatively be deduced from the functoriality of knot Floer homology with respect to cobordisms and its grading shift formula \cite{ZemkeLink,ZemkeGrading}, which appeared during the course of our work.  Our proof is significantly simpler, avoiding as it does most of the numerous subtleties involved with establishing the full functoriality of knot Floer homology. It also establishes the inequalities for rationally null-homologous knots, though Zemke's grading shift formula should readily adapt to this setting.  Furthermore, it allows us to correct and clarify an issue in the literature (see Remark \ref{circle-error}). We cannot, however, recover some of the beautiful applications that the full functoriality of knot Floer homology has recently afforded  \cite{2018arXiv180409589J,2018arXiv181009158J,2019arXiv190204050Z,2019arXiv190305772M,2019arXiv190402735J}.

\subsection{Applications}
Our primary motivation for pursuing the general relative adjunction inequality stems from a number of topological applications pursued here and in subsequent papers. In the remainder of the introduction, we briefly describe some of these applications, and conclude with an outline of the paper.

Theorem \ref{thm:relativeadjunction} allows us to define concordance invariants of links, using \ons's ``knotification" procedure \cite[Subsection 2.1]{Knots}.  For an $|L|$ component link, these invariants are indexed by elements in the cohomology of an $|L|-1$ dimensional torus,  $H^*(\mathbb{T}^{|L|-1})$, a graded $\F$-vector space  of rank $2^{|L|-1}$ corresponding to the Floer homology of $\#^{|L|-1}S^1\times S^2$.  The natural grading is shifted down so that the bottom graded summand lives in degree $\frac{1}{2}{(1-|L|)}$. Each element of $H^*(\mathbb{T}^{|L|-1})$ gives rise to a link concordance invariant which generalizes the \os-Rasmussen invariant for knots.  We summarize our results in the following theorem, whose parts are proved in Section \ref{subsec:links}.  In the case that $|L|=1$, and $L$ is a knot, then we replace $H^*(\mathbb{T}^{|L|-1})$ with reduced cohomology, a vector space of rank one. Restricting to the case of knots in the $3$-sphere ($|L|=1$ and $Y=S^3$) the theorem recovers previously known results on  $\tau(K)$, but for a  general manifold the results are new regardless of the number of link components.

\begin{thm}\label{thm:links} Let $L\subset  Y$ be a rationally null-homologous knot or link with $|L|$ components.  Then,  given any  non-trivial element $\alpha\otimes \Theta\in  \HFa(Y)\otimes_{\F} H^*(\mathbb{T}^{|L|-1})$, we have an invariant  $\tau_{\alpha\otimes\Theta}(Y,L)$ satisfying:
\begin{enumerate}[leftmargin=*, label=(\alph*)]
\item {\em{\bf Corollary \ref{cor:ConcordanceInvariance}} (Concordance invariance){\bf.}} \label{thm:links1} If $L$ is concordant to $L'$ in $Y\times[0,1]$, then  
\begin{equation*}
\tau_{\alpha\otimes\Theta}(Y,L)=\tau_{\alpha\otimes\Theta}(Y,L').
\end{equation*}

\item {\em{\bf Corollary \ref{cor:Crossingchange} }(Crossing change inequalities){\bf.}}\label{thm:links-crossing} If $L_-,L_+\subset Y$ differ at a single crossing, which is positive  in $L_+$ and negative in $L_-$, then
\[ \tau_{\alpha\otimes \Theta}(Y,L_-)\le \tau_{\alpha\otimes \Theta}(Y,L_+)\le \tau_{\alpha\otimes\Theta}(Y,L_-)+1.\]

\item {\em{\bf Proposition \ref{prop:SlicegenusBounds-links}} (Slice-genus bounds){\bf.}} \label{thm:links2} If  $\Sigma\subset Y\times[0,1]$ is a smoothly embedded oriented surface with boundary $ L\subset Y\times\{1\}$, then its Euler characteristic satisfies \[2|\tau_{\alpha\otimes\Theta}(Y,L)|\le |L|-\chi(\Sigma).\] 

\item {\em {\bf Proposition \ref{prop:Monotonicity}} (Monotonicity){\bf.}}\label{thm:links-monotoncity} If $\Theta'=\iota_x(\Theta)$, where  $\iota_x$ denotes the interior product with a class $x\in H_1(\mathbb{T}^{|L|-1})$, then 
\[ \tau_{\alpha\otimes \Theta'}(Y,L)\le \tau_{\alpha\otimes \Theta}(Y,L)\le \tau_{\alpha\otimes\Theta'}(Y,L)+1.\]
In particular, if $L\subset S^3$ and $\tautop(L)$ and $\taubot(L)$ denote the invariants corresponding to the unique elements in $H^*(\mathbb{T}^{|L|-1})$ of maximal and minimal grading, respectively, then
\[ \taubot(L)\le \tau_{\Theta}(L)\le \tautop(L)\le \taubot(L)+|L|-1.\]

\item {\em {\bf Theorem \ref{thm:linkgenus}} (Definite 4-manifold bound){\bf.}}\label{thm:links-negdef} Let $W$ be a smooth, oriented 4-manifold with $b_2^+(W)=b_1(W)=0$, and $\partial W=S^3$. If $\Sigma\subset W$ is a smoothly embedded oriented surface with boundary a link $L\subset \partial W$, then
$$2\tau_\Theta(L)+[\Sigma]^2+|[\Sigma]|\leq |L|-\chi(\Sigma).$$
Here $|[\Sigma]|$ is the $L_1$-norm of the homology class $[\Sigma]\in H_2(W,\partial W)\cong H_2(W)$.

\item {\em {\bf Theorem \ref{thm:altLinks}} (Alternating links){\bf.}} \label{thm:links-alternating} Suppose $L\subset S^3$ is an alternating link of $|L|$ components, and $\Theta\in \HFa(\#^{|L|-1}S^1\times S^2)$ is a class with grading $k$.   Then $\tau_\Theta(L)= k-\frac{\sigma}{2}$, where $\sigma(L)$ is the signature. In particular, 

\begin{equation*}
\tautop(L)= \frac{|L|-\sigma(L)-1}{2} \;\;\;\; \text{ and } \;\;\;\; \taubot(L)=\frac{-|L|-\sigma(L)+1}{2}.
\end{equation*}

\item {\em{\bf Proposition \ref{prop:quasipositive}} (The Local Thom and Milnor Conjectures){\bf.}}  \label{thm:milnor}  For a link $L\subset S^3$ bounding a complex curve in $B^4\subset \C^2$ or, equivalently, possessing a quasipositive braid representative, we have
\begin{equation*}
\ \ \ \ \ \tautop(L)= g_4(L):= \mathrm{min} \left\{\ \frac{|L|-\chi(\Sigma)}{2} \;\Big| \;\Sigma\subset B^4, \mathrm{smooth, oriented, with}\ \partial \Sigma=L\right\}, 
\end{equation*}
and the minimum is realized by any complex curve in $B^4$  bounded by $L$.

\item {\em{\bf Theorem \ref{thm:SQP}} (Detection of fibered strongly quasipositive links){\bf.}}  \label{thm:links-SQP} If  $L\subset S^3$ is fibered, then $L$ is strongly quasipositive if and only if 
\[\tautop(L)= g_4(L) = g_3(L),\]
where $g_3(L)$ is defined analogously to $g_4(L)$, but with surfaces embedded in $S^3$.
\end{enumerate}
\end{thm}

During the course of our work,  other definitions of $\tau$ were formulated for links in $S^3$ using grid homology by \os-Stipsicz \cite{GridBook} and Cavallo \cite{Cavallo}, respectively, and several of the properties and applications listed above have been established in that context.  We compare the specialization of our invariants to links in $S^3$  with theirs in Subsection \ref{subsec:comparison} and prove that $\tau_{top}(L)$ equals Cavallo's invariant $\tau(L)$ and \os-Stipsicz's invariant $\tau_{max}(L)$.  See Theorem \ref{thm:tausequal}.  

An additional  feature of  our  invariants is a Bennequin type inequality.  In the special case of links in the 3-sphere, it states that:
\[\tb(\mathcal{L})+\rot(\mathcal{L})+|L|-1\leq 2\tautop(L)-1,\]
for any Legendrian representative of $L$ in the standard contact structure on $S^3$.  Combined with the slice-genus bound for $\tau_{top}(L)$ above, we obtain a refinement of Rudolph's well-known slice-Bennequin bound \cite{Rudolph1993}.  We will establish the Bennequin bound for $\tautop(L)$ in  \cite{4Dtight}, where we use the relative adjunction inequality in combination with a Bennequin bound from \cite{tbbounds} to prove a slice-Bennequin inequality for contact manifolds with non-vanishing contact invariants.

In another direction, we can extend Theorem \ref{thm:relativeadjunction}  to the situation where $K_1$ and $K_2$ are only rationally homologous, i.e.\ some multiples of their respective homology classes agree in $H_1(W)$.  This allows us to study an analogue of the slice genus for knots that don't bound {\em any} surface in a given 4-manifold with boundary.  This extension is motivated by the case of a rational homology 3-sphere times an interval, and the study of a rational analogue of slice genus for knots in a non-trivial homology class.   Ni and Wu \cite{Ni-Wu} proved the remarkable result that Floer simple knots in $L$-spaces minimize the rational Seifert genus of any knot in their homology class.  The general version of Theorem \ref{thm:relativeadjunction} allows us to show  that Floer simple knots moreover have rational slice genus equal to their rational Seifert genus.  The proof of this extension is more technical, and will be taken up in a forthcoming paper.\\

\noindent {\bf Outline:} The paper is organized as follows. In Section \ref{sec:background},  we define and recall some elementary properties of our invariants, and compute them for a simple example. In Section \ref{sec:tools}, we outline the strategy for our proof of Theorem \ref{thm:relativeadjunction} and establish some key tools for its implementation.  Specifically,  we extend the K\"unneth theorem for the Floer homology of connected sums of 3-manifolds to cobordisms,   prove a vanishing result  for the cobordism maps in the presence of a homologically essential surface, and describe the relationship between $\tau_\beta$ invariants and the maps on Floer homology induced by $2$-handle cobordisms.  Section \ref{sec:mainproof} proceeds with the proof of Theorem \ref{thm:relativeadjunction} and Section \ref{sec:applications} includes applications and examples, including the results listed in Theorem \ref{thm:links}.   \\

\noindent {\bf Acknowledgements:}  This paper evolved over many years, and enjoyed the benefit of interest and input from a number of people, whom we warmly thank: John Baldwin, Inanc Baykur, Alberto Cavallo, Georgi Gospodinov,  Eli Grigsby, Miriam Kuzbary, Adam Levine,  Chuck Livingston,  Tom Mark, Peter Ozsv{\'a}th, Olga Plamenevskaya,   Danny Ruberman, Sucharit Sarkar, Linh Truong, Zhongtao Wu, and Ian Zemke.  KR  also  thanks Bryn Mawr College for hosting her as a research associate, and MH thanks MSRI and AIM for the supportive environments provided for portions of this work. Finally, we thank the referee for a thorough reading and many helpful suggestions.
  \\

\section{Background on Heegaard Floer Theory}\label{sec:background} In this article, all manifolds are assumed to be oriented and knots are assumed to be both oriented and rationally null-homologous.   Knots and $3$-manifolds will also be assumed to be pointed, though we will  typically omit this structure from the notation and discussion. Similarly, cobordisms between pointed  $3$-manifolds will implicitly be equipped with an oriented path between basepoints.  The role of the basepoints and paths is essential for the functoriality of Heegaard Floer homology, see \cite{JTZ,ZemkeGraphCob}.

We assume the reader has a basic familiarity with Heegaard Floer theory and knot Floer homology at the level of \cite{HolDisk, Knots, FourBall}. This article is, in a sense, the sequel of \cite{tbbounds}. Here, however, we use the more general construction of knot Floer homology for rationally null-homologous knots, and in the first subsection we recall and clarify the structure of the theory in this setting; see \cite{Hedden-Levine-surgery} for further details.  Having done this, we turn to the definition and elementary properties of the generalized $\tau$ invariants, which we collect in Subsection  \ref{subsec:taudef}.  We then compute a simple example of our invariants for a knot in $\SoneStwo$, which will ground the discussion moving forward.

\subsection{Heegaard Floer homology and the knot filtration} 

In \cite{HolDisk}, Ozsv\'ath and Szab\'o define a complex $\CFinf(Y)$ associated to a pointed Heegaard diagram $(\Sigma,\alphas,\betas,w)$ for a (pointed) 3-manifold $Y$. This complex is generated over $\F=\Z/2\Z$ by elements $[\x,i]$, where $\x\in\Ta\cap\Tb$ is an intersection point between Lagrangian tori specified by the Heegaard curves in the $g$-fold symmetric product of $\Sigma$, and $i$ is an integer. The differential is given by $$\partial [\x,i]=\sum_{y\in\Ta\cap\Tb}\sum_{\substack{ \phi\in\pi_2(\x,\y) \\ \mu(\phi)=1}}\#\widehat{\mathcal{M}}(\phi)[\y,i-n_w(\phi)],$$ where $\#\widehat{\mathcal{M}}(\phi)$ denotes the number of points, modulo two, in the unparameterized moduli space of pseudo-holomorphic disks connecting $\x$ to $\y$ in the homotopy class $\phi$, and $n_w(\phi)$ is the algebraic intersection number of such a disk with the complex codimension one subvariety $V_w$ of the symmetric product consisting of unordered tuples of points that contain the basepoint $w$.  This complex is the Lagrangian Floer complex 
 of the pair $(\Ta,\Tb)$ with a twisted coefficient system coming from a distinguished $\Z$ summand of the fundamental group of the path space.  As such, it 
 has the structure of a free $\F[\Z]=\F[U,U^{-1}]$ module, determined by the action of the generator: $U\cdot[\x,i]=[\x,i-1]$.  The complex admits a direct sum decomposition indexed by $\SpinC$ structures on $Y$, and we denote the summand corresponding to a $\SpinC$ structure $\spinc$ by $\CFinf(Y,\spinc)$ 
 
Positivity of intersections between complex subvarieties of complementary dimension implies that pseudo-holomorphic disks intersect $V_w$ positively, provided that the family of almost complex structures used in defining the boundary operator is integrable in a neighborhood of $V_w$.  Hence $\CFinf(Y,\spinc)$ is naturally filtered by the $i$ parameter of the generators $[\x,i]$.  
Ozsv\'ath and Szab\'o prove that the filtered homotopy type of $\CFinf(Y,\spinc)$ is an invariant of the pair $(Y,\spinc)$ and, as a consequence, they obtain a number of  invariants derived from this homotopy type.  Most notably, the complex $\CFm(Y,\spinc)$ is the subcomplex generated by elements $[\x,i]$ where $i<0$ and $\CFp(Y,\spinc)$ is the resulting quotient complex. Also featuring prominently in the theory is $\CFa(Y,\spinc)$, defined as the kernel complex of the chain map $U:\CFp(Y,\spinc)\to\CFp(Y,\spinc)$, or, alternatively, as the associated graded complex at filtration level $0$. This ``hat" complex  is the central object of study  in the present article. The homologies of these complexes are denoted $\HFinf(Y,\spinc)$, $\HFm(Y,\spinc)$, $\HFp(Y,\spinc)$ and $\HFa(Y,\spinc)$, respectively.

In \cite{Knots} and \cite{RationalSurgeries} Ozsv\'ath and Szab\'o show that an oriented rationally null-homologous knot $K\subset Y$ gives rise to an additional filtration of the above complexes.  The filtration can be interpreted geometrically in terms of relative $\SpinC$ structures on the knot complement. To understand this, let $(\Sigma,\alphas,\betas,w,z)$ be a doubly pointed (admissible) Heegaard diagram for $(Y,K)$.  The splitting of the complex $\CFinf(Y)$ along $\SpinC$ structures is defined by a map $\spinc_w(-):\Ta\cap\Tb \to \SpinC(Y)$. In \cite{RationalSurgeries}, Ozsv\'ath and Szab\'o refine this map to take values in relative $\SpinC$ structures, defined as $\SpinC$ structures on the knot complement with prescribed restriction to the boundary.  The refined map, denoted $\spinc_{w,z}(-):\Ta\cap\Tb\to\SpinC(Y,K)$, fits in a commutative diagram
$$
\xymatrix{
\Ta\cap\Tb \ar[r]^-{\spinc_{w,z}}\ar[dr]_{\spinc_{w}}& \SpinC(Y,K)\ar[d]^{G_{Y,K}}\\
& \SpinC(Y).
}
$$
where $G_{Y,K}:\SpinC(Y,K)\to \SpinC(Y)$ is a filling map described in Section 2 of \cite{RationalSurgeries}. The preimage of a $\SpinC$ structure $\spinc$ under the map $G_{Y,K}$ is endowed with a free and transitive $\Z$ action by the subgroup of $H^2(Y,K)$ generated by the Poincar{\'e} dual of the class of the meridian $\mu_K$. This action identifies the fibers $G_{Y,K}^{-1}(\spinc)$ with $\Z$ and, here again, positivity of intersections (now between pseudo-holomorphic disks and $V_z$) implies that the complex is relatively filtered by the additional $\Z$ parameter.  We denote the corresponding $\Z\oplus\Z$-filtered complex by $\CFKinf(Y,K,\spinc)$.

It is  convenient and useful to turn the relative $\Z$ filtration induced by the affine identification $\Z\cong G_{Y,K}^{-1}(\spinc)$ into an absolute filtration.  This identification, called the {\em Alexander filtration}, has the added benefit of offering a comparison between relative $\SpinC$ structures in different orbits of the action by $\PD[\mu]$ or, equivalently, different fibers of $G_{Y,K}$. To do this, we use a rational Seifert surface $S$ for $K$.

 \begin{defn} A \emph{rational Seifert surface} for a knot $K\subset Y$ of order $q$  is a compact, oriented surface $S$ with boundary, along with a map $S\to Y$ that is an embedding on the interior of $S$ and whose restriction to $\partial S$ is a map $\partial S \to K$, which is a covering map of degree $q$. We let $S$ denote the singular surface in $Y$ arising as the image of the defining map.\end{defn}
 
\noindent A rational Seifert surface gives rise to a properly embedded surface in the  complement of $K$ that intersects the boundary of its tubular neighborhood in a cable link.  One can alternatively define rational Seifert surfaces in these terms. See also \cite{Calegari-Gordon, Baker-Etnyre}.

Now, define the \emph{Alexander grading} of a relative $\SpinC$-structure $\xi\in\SpinC(Y,K)$ by $$A_{Y,K,[S]}(\xi)=\frac{\langle c_1(\xi),[S]\rangle +[\mu]\cdot[S]}{2[\mu]\cdot[S]}\in \Q.$$ Here $[\mu]\cdot[S]$ denotes the intersection pairing between $H_1(Y\setminus K)$ and $H_2(Y,K)$ induced by Lefschetz duality and excision, and $c_1(\xi)\in H^2(Y,K)$ is the relative Chern class of the relative $\SpinC$ structure.  For a generator $\x\in\Ta\cap\Tb$, define $$A_{w,z}(\x)=A_{Y,K,[S]}(\spinc_{w,z}(\x)).$$
We write $\CFKinf(Y,[S],K,\spinc)$ to denote $\CFKinf(Y,K,\spinc)$ with absolute filtration coming from $A_{Y,K,[S]}$.
 
\begin{remark}\label{rmk:dependence}
A couple of remarks are in order.  First, if $Y$ is not a rational homology 3-sphere, the Alexander grading depends on the relative homology class of the chosen rational Seifert surface $S$ for $K$, but only up to an overall shift, given by $\frac{1}{2[\mu]\cdot[S]}$ times
\[  \langle c_1(\xi), [S]-[S']\rangle = \langle c_1(\xi), i_*([S-S'])\rangle= \langle i^*c_1(\xi), [S-S']\rangle=\langle c_1(G_{Y,K}(\xi)),[S-S']\rangle \in 2\Z.\]
Here, we use that $[S]-[S']\in H_2(Y,K)$ is in the image of the inclusion induced map $i_*$ from $H_2(Y)$, together with naturality of relative Chern classes.  

Second, there are different conventions for the definition of the Alexander grading in the literature \cite{NiThurston, Ni-Wu, rationalcontact, Hedden-Levine-surgery}. Ours is consistent with \cite{Hedden-Levine-surgery}.
\end{remark}

\subsection{$\tau$ invariants and their properties}\label{subsec:taudef} Consider the complex $\CFa(Y,\spinc)$ equipped with its Alexander filtration. For $r\in\Q$ there is a subcomplex generated by $\x\in \Ta \cap \Tb$ whose Alexander grading is less than or equal to $r$:
 $$\Filt_r(Y,[S],K) =\underset{\{ \x\ |\ A_{w,z}(\x)\leq r\}}\bigoplus \!\!\!\!\!\F\langle\x\rangle.$$ 
This subcomplex includes into $\CFa(Y,\spinc)$ by $\iota_r:\Filt_{r}(Y,[S],K)\xhookrightarrow{}\CFa(Y,\spinc)$.

\begin{defn}
For a nontrivial class $\alpha$ in  $\HFa(Y,\spinc)$, $$\tau_{\alpha}(Y,[\SSurf],K)=\min \{ r\in \Q \;| \;\alpha\in \Image(I_{r})\}$$ where $I_{r}$ is the map induced on homology by $\iota_r$. 
\end{defn}

We write simply $\tau_{\alpha}(Y,K)$ or $\tau_{\alpha}(K)$ when the context is clear. Note that we could equivalently define $\tau_{\alpha}(Y,K)$ as the minimum Alexander grading of any cycle homologous to $\alpha$.  Using a duality pairing on Floer homology we also define $\tau^*_{\varphi}(Y,K)$:
\begin{defn}
For a nontrivial class  $\varphi$ in $\HFa^*(Y,\spinc)\cong\HFa_*(-Y,\spinc)$, $$\tau^*_{\varphi}(Y,[\SSurf],K)=\min\{ r\in \Q \;|\; \text{there exists}\; \beta \in \Image(I_r) \text{ such that } \langle \varphi,\beta \rangle\neq 0\}$$ where $\langle -, -\rangle$ denotes the pairing on Floer homology between $\HFa^*(Y,\spinc)$ and $\HFa_*(Y,\spinc)$. For details about this pairing see Section 2 of \cite{tbbounds}.
\end{defn}
\noindent The quantities defined above are related in the following way:
\begin{prop}[Duality]   \cite[Proposition $28$]{tbbounds}\label{prop:dualtau}  Let  $\beta$ be a nontrivial class in $\HFa(-Y,\spinc)$ then
$$\tau_{\beta}(-Y,[\SSurf],K)=-\tau^*_\beta(Y,[\SSurf],K).$$
\end{prop}

In addition, $\tau_{\alpha}$ and $\tau^*_{\varphi}$ are additive under connected sum. Specifically, let $K_1$ and $K_2$ be knots in 3-manifolds $Y_1$ and $Y_2$, respectively, and let $K_1\# K_2$ denote their connected sum inside $Y_1\#Y_2$.

\begin{prop}[Additivity]  \cite[Proposition $3.2$]{FourBall}\label{prop:Additivity}  For any pair of non-trivial Floer classes $\alpha_i\in \HFa(Y_i,\spinc_i)$,  
$$\tau_{\alpha_1\otimes \alpha_2}(Y_1\#Y_2, K_1\# K_2)=\tau_{\alpha_1}(Y_1,K_1)+\tau_{\alpha_2}(Y_2,K_2),$$
where $\alpha_1\otimes \alpha_2$ specifies a Floer class for the connected sum under the isomorphism
$$\HFa(Y_1,\spinc_{1})\otimes_{\F} \HFa(Y_2,\spinc_{2}) \cong \HFa(Y_1\# Y_2,\spinc_{1}\#\spinc_{2}).$$
Similarly, for $\tau^*$,
$$\tau^*_{\varphi_1\otimes \varphi_2}(Y_1\#Y_2, K_1\# K_2)=\tau^*_{ \varphi_1}(Y_1,K_1)+\tau^*_{ \varphi_2}(Y_2,K_2),$$
for any pair of non-trivial classes $ \varphi_i\in \HFa^*(Y_i,\spinc_i)$.
\end{prop}

\begin{remark} \label{rmk:Seifertsum} For null-homologous knots, a Seifert surface for $K_1\#K_2$ is given by the boundary sum $\SSurf_1\natural\SSurf_2$ of Seifert surfaces $\SSurf_1$ and $\SSurf_2$ used for $K_1$ and $K_2$, respectively. More generally, if $K_1$ and $K_2$ are rationally null-homologous of orders $q_1$ and $q_2$ respectively, a rational Seifert surface for $K_1\#K_2$ can be constructed as a band sum of $\frac{\lcm(q_1,q_2)}{q_1}$ copies of $\SSurf_1$ and $\frac{\lcm(q_1,q_2)}{q_2}$ copies of $\SSurf_2$ along $\lcm(q_1,q_2)$ bands. 
\end{remark}

\noindent Since $\tau_\alpha$ and $\tau^*_\varphi$ are defined in terms of the Alexander filtration, whose homotopy type is an invariant of the knot, they are also invariant in an appropriate sense.  We clarify this with the following proposition.  

\begin{prop}[Functoriality] \label{prop:naturality} Let $f:(Y,K,w)\rightarrow (Y',K',w')$ be a diffeomorphism of pointed knots, and $\alpha\in \HFa(Y,w)$ be a non-trivial Floer homology class.  Then 
\[\tau_{\alpha}(Y,[S],K)= \tau_{f_*(\alpha)}(Y',[S'],K'),\]
where $f_*(\alpha)$ is the image of $\alpha$ under the diffeomorphism-induced map on Floer homology $f_*:\HFa(Y,w)\rightarrow \HFa(Y',w')$, and $[S']=f_*[S]$.
\end{prop}
\begin{proof} This is a consequence of the naturality of knot Floer homology under diffeomorphisms established by Juh\'asz-Thurston-Zemke \cite{JTZ}. More precisely, that article shows how to use the Heegaard Floer construction to associate a transitive system of groups to a knot complement, regarded as a sutured manifold with two parallel meridional sutures, and describes how diffeomorphisms act on this invariant \cite[Definition 2.42]{JTZ}.  The transitive system and diffeomorphism action can be lifted to the homotopy category, and indeed to the $\Z\oplus\Z$-filtered homotopy category.  See \cite[Proposition 2.3]{HM} for details on lifting the ``infinity" invariant of a pointed 3-manifold to the $\Z$-filtered homotopy category, where the filtration is given by powers of $U$.  The extension to the $\Z\oplus\Z$-filtered homotopy category follows in a similar  manner.  Specializing to the induced $\Z$-filtration of the hat complex induced by the pointed knot, from whence the $\tau$ invariants are derived, we obtain the claimed result.
\end{proof}

In general, knowing $\tau_\alpha(Y,K)$ and $\tau_\beta(Y,K)$ does not determine $\tau_{\alpha+\beta}(Y,K)$.  The following proposition, however, follows easily from the definition.

\begin{prop}[Subadditivity] \label{prop:sum}For  classes $\alpha,\beta\in \HFa(Y)$ and any knot $K\subset Y$, we have
\[  \tau_{\alpha+\beta}(Y,K)\le \mathrm{max}\{\tau_\alpha(Y,K), \tau_\beta(Y,K)\}\]
\end{prop}
\noindent When using this, one should extend the definition so that $\tau_\alpha(K)=-\infty$ when $\alpha=0$.
\subsection{Calculations for knots in $\#^\ell\SoneStwo$}
We briefly describe our invariants for knots in $\#^\ell\SoneStwo$ and provide a simple example calculation that will be used as a guide throughout the paper.

In \cite[Section 9]{HolDisk} Ozsv\'ath and Szab\'o compute $\HFa(\SoneStwo)$ and show that it is  generated by two elements  in the $\SpinC$ structure with trivial Chern class, one of Maslov grading $\frac{1}{2}$ and one of Maslov grading $-\frac{1}{2}$. Let $\theta_{+}$ and $\theta_{-}$ denote the generators of highest and lowest Maslov gradings, respectively. The K\"unneth formula for the Heegaard Floer homology of connected sums \cite[Theorem 1.5]{HolDiskTwo} implies that $\HFa(\#^\ell\SoneStwo)$ is generated by $\ell$-fold tensor products $\theta_{\epsilon_1}\otimes\ldots\otimes \theta_{\epsilon_\ell}$ where each $\epsilon_i\in\{+,-\}$. 

\begin{example} \label{Whitehead}
\begin{figure}
\begin{subfigure}{.25\textwidth}
\centering
\includegraphics[height=1in]{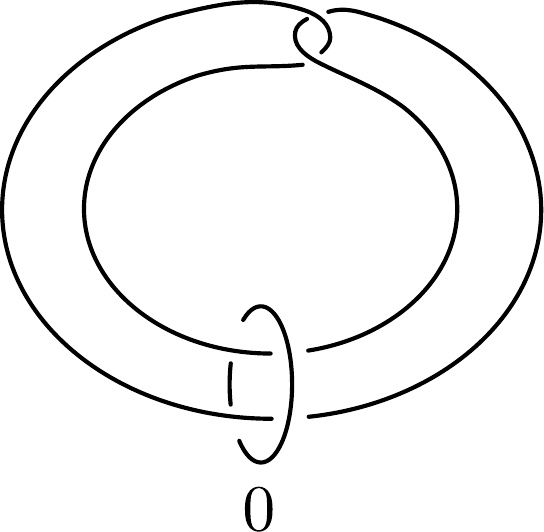}
\caption{\centering$Wh^+\subset \SoneStwo$.}
\label{fig:Wh+}
\end{subfigure}
\quad
\begin{subfigure}{.3\textwidth}
\centering
 \includegraphics[height=1in]{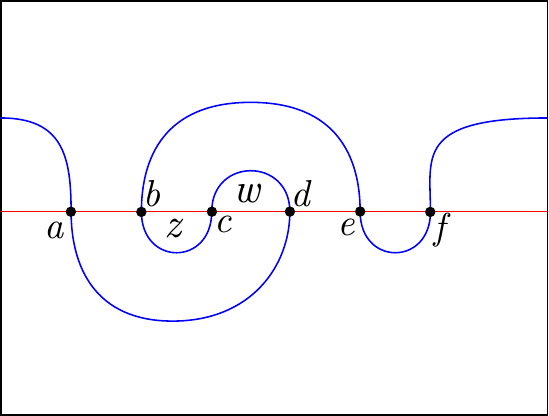} 
 \caption{\centering An admissible diagram for $Wh^+ \subset S^1\times S^2$.}
 \label{fig:Whitehead}
 \end{subfigure}
 \quad
\begin{subfigure}{.3\textwidth}
\centering
\includegraphics[height=.75in]{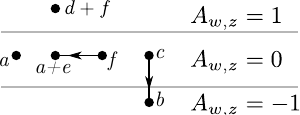}
\caption{\centering$\CFa(S^1\times S^2)$ with grading $A_{w,z}$.}
\label{fig:Agrading}
\end{subfigure}
\caption{}
\end{figure}

Given a knot $K$ in $\#^\ell\SoneStwo$, each generator of Floer homology has a corresponding invariant: $\tau_{\theta_{\epsilon_1}\otimes\ldots\otimes \theta_{\epsilon_\ell}}(\#^\ell\SoneStwo,K)$. Since the Maslov grading is additive under tensor products, there is a unique generator of highest Maslov grading, which we call $\Theta_{top}=\theta_{+}\otimes \ldots \otimes \theta_{+}$ and a unique generator of lowest Maslov grading, $\Theta_{bot}=\theta_{-}\otimes \ldots \otimes \theta_{-}$. We write $\tautop(\#^\ell\SoneStwo,K)$ for the invariant associated to $\Theta_{top}$ and $\taubot(\#^\ell\SoneStwo,K)$ for the invariant associated to $\Theta_{bot}$.

Figure \ref{fig:Wh+} is a diagram of the positively clasped Whitehead knot $Wh^+$ in $\SoneStwo$ and Figure \ref{fig:Whitehead} is an admissible, doubly pointed Heegaard diagram for this knot. This knot is an example of a $(1,1)$ knot and therefore its Floer Homology can be calculated combinatorially from the Heegaard diagram -- see, for instance, \cite[Section 6.2]{Knots} or \cite{Goda}. The relative Alexander gradings can be computed, for $\phi\in\pi_2(\x,\y)$, by $$A_{w,z}(\x)-A_{w,z}(\y)=n_z(\phi)-n_w(\phi).$$ 
Moreover, all holomorphic disks (there are eight) are determined by the Riemann mapping theorem.  After a filtered change of basis we obtain the complex in Figure \ref{fig:Agrading}. The absolute Alexander gradings of the generators can be computed using  \cite[Proposition 1.3]{Hedden-Levine-surgery} or, alternatively, by requiring the associated graded homology to be symmetric about $A_{w,z}=0$.  Examining disks in the diagram shows that $[d+f]$ has higher relative Maslov grading than $[a]$ -- for instance, in the diagram we see a disk $\phi$ from $d$ to $a$ that passes over the $z$ basepoint thus the relative grading is given by $gr(d,a)=\mu(\phi)-2n_w(\phi)=1$. Thus, $\tau_{top}(Wh^+)=1$ and $\tau_{bot}(Wh^+)=0$.
 \end{example}

\section{Collecting the main ingredients}\label{sec:tools}
Given a cobordism from $Y_1$ to $Y_2$ containing a properly embedded surface $\Sigma$, the strategy for proving Theorem \ref{thm:relativeadjunction} is to ``cap" both ends with $2$-handle cobordisms, attached along neighborhoods of the knots  $-K_1\sqcup K_2$ on the boundary of $\Sigma$.  
We then use our assumption that the cobordism map sends $\alpha$ to $\beta$ in conjunction with a 
 result indicating that the $\tau$ invariants control the maps on Floer homology induced by $2$-handle cobordisms with sufficiently large framings. This yields conditions, expressed in terms of the difference $\tau_\beta(K_2)-\tau_\alpha(K_1)$, for the capped cobordism map to be non-trivial.

We can factor the capped cobordism, however, through a neighborhood of its incoming end joined to the closed surface one gets from $\Sigma$ by capping $-K_1$ and $K_2$ with the core disks of the $2$-handles.   
 We then employ a vanishing result  for the cobordism map associated to this factorization, expressed in terms of the genus of $\Sigma$,   in conjunction with the  conditions for its non-triviality above.  This bounds the difference of $\tau$ invariants by the genus of $\Sigma$ and the homological terms appearing in the relative adjunction inequality.   

In this section, we pave the way for employing the strategy outlined above by establishing the requisite technical tools.   The first is a product formula, Theorem \ref{thm:Kun}, for the cobordism maps associated to 4-manifolds obtained by a  surgery operation along properly embedded paths which we call the {\em arc sum}.   The second is the vanishing result for cobordisms containing a homologically essential surface, Theorem \ref{thm:cobordismbound}.   Finally,  we describe the manner in which the $\tau$ invariants constrain the behavior of cobordism maps associated to 2-handle attachments.  This is the content of  Proposition \ref{prop:FourDInterp}.

\subsection{Splittings of $\SpinC$ structures}  It will be useful throughout to understand when a $\SpinC$ structure on a 4-manifold can be determined by its restrictions to pieces glued along a separating 3-manifold.

\begin{lemma}\label{lemma:splitting}
Let $W$ be a 4-manifold, and suppose that $Y$ is a separating 3-manifold embedded in $W$ such that $W=W_1\cup_{Y}W_2$. Each $\SpinC$ structure $\spinct$ on $W$ has restrictions $\spinct_1=\spinct|_{W_1}$ and $\spinct_2=\spinct|_{W_2}$. If the map $$(\iota_1)^*-(\iota_2)^*:H^1(W_1)\oplus H^1(W_2)\to H^1(Y)$$ in the Mayer-Vietoris sequence is surjective, then $\spinct$ is uniquely determined by its restrictions to $W_1$ and $W_2$. That is, we may unambiguously write $\spinct=\spinct_1\#\spinct_2$.
\end{lemma}
\begin{proof}
$\SpinC(-)$ is an affine $H^2(-;\Z)$-set. Furthermore, restriction of $\SpinC$ structures to codimension zero submanifolds is in affine correspondence with restriction of cohomology classes. Consider the Mayer-Vietoris sequence: 
$$
\begin{CD}
\cdots  @>{(\iota_1)^*-(\iota_2)^*}>>H^1(Y) @>{\delta}>> H^2(W)@>{j^*}>>H^2(W_1)\oplus H^2(W_2)@>>>H^2(Y)\cdots 
\end{CD}
$$
If $(\iota_1)^*-(\iota_2)^*$ is surjective, then $\delta\equiv0$ and $j^*$ is injective. Thus, each element of $H^2(W)$ has a unique decomposition as a class in $H^2(W_1)\oplus H^2(W_2)$.  It follows from the affine identifications that the same holds for $\SpinC$ structures.
\end{proof}

One particular instance of Lemma \ref{lemma:splitting} is at the heart of our applications. Consider  $$W=W_1\cup_Y W_2$$ where $W_2= W_\lambda(K)$ is the 4-manifold obtained by adding a 2-handle to $Y\times [0,1]$ along a rationally null-homologous knot $K$ with framing $\lambda$. Consider the exact sequence in homology associated to the pair $(W_\lambda(K), Y)$, $$H_2(W_\lambda(K), Y)\xrightarrow{\partial} H_1(Y)\to H_1(W_\lambda(K))\to 0.$$
The boundary map sends the generator of $H_2(W_\lambda(K),Y)\cong\Z$ to $[K]$. Since $K$ is rationally null-homologous, the image of $\partial$ is contained in the torsion subgroup of $H_1(Y)$. This implies that the map $\Hom(H_1(W_\lambda(K));\Z)\to \Hom(H_1(Y);\Z)$ is surjective and, therefore, the map $H^1(W_\lambda(K))\to H^1(Y)$ is as well. Thus, Lemma \ref{lemma:splitting} applies to $W$.

\subsection{A  K{\"u}nneth theorem for cobordisms}

For both the proof of Theorem \ref{thm:relativeadjunction} and for the vanishing result for cobordism maps in the next subsection, it will be useful to have a 4-dimensional analogue of \ons's  formula for the Floer homology of a connected sum of 3-manifolds \cite[Theorem 1.5]{HolDiskTwo}.  

A cobordism between pointed $3$-manifolds $(Y_1,w_1)$ and $(Y_2,w_2)$ is a pair  $(W,\Gamma)$ consisting of a cobordism and a smooth properly embedded path $\Gamma$ from $w_1$ to $w_2$. Let $(W,\Gamma)$ be such a cobordism, and let $(W',\Gamma')$ be another cobordism between pointed 3-manifolds $(Y_1',w_1')$ and $(Y_2',w_2')$. Define the \emph{arc sum} of $(W,\Gamma)$ and $(W',\Gamma')$ to be the 4-manifold $$W \otimes W':= W\setminus \nu(\Gamma) \underset{S^2\times I} \cup W'\setminus \nu(\Gamma')$$ obtained by removing tubular neighborhoods of the paths and gluing the remainder using an orientation reversing diffeomorphism of the resulting  $S^2\times I$ in their boundaries. The arc sum is naturally a cobordism between the  pointed 3-manifolds $Y_1\# Y'_1$ and $Y_2\#Y'_2$, endowed with a proper arc along the $S^2\times I$ where the identification is made.

We have the following K{\"u}nneth-type theorem for cobordism maps.

\begin{theorem}[Product formula for arc sums]\label{thm:Kun}
Given  $\SpinC$-cobordisms $(W,\spinct)$ and $(W',\spinct')$ from   $(Y_1,\spinc_1)$ to $(Y_2,\spinc_2)$ and $(Y'_1,\spinc'_1)$ to $(Y'_2,\spinc'_2)$, respectively, equipped with properly embedded arcs $\Gamma\subset W$, $\Gamma'\subset W'$, we have a commutative  diagram:

$$
\begin{CD}
	\HFa(Y_1,\spinc_1)\otimes\HFa(Y'_1,\spinc'_1) 
	@>{   {  F_{W,\spinct}  } \otimes {F_{ W',{\spinct'}  }  }}>>                  
	\HFa(Y_2,\spinc_2)\otimes\HFa(Y'_2,\spinc'_2) \\
	@V{\cong}VV    \   @V{\cong}VV \\
		\HFa(Y_1\# Y'_1,\spinc_1\#\spinc'_1) @>{  \ \ \ F_{W\otimes W',{\spinct\#\spinct'}}\ \   }>> \HFa(Y_2\#Y'_2,\spinc_2\#\spinc'_2)\\
\end{CD}
$$
where $W\otimes W'$ is the arc sum of $W$ and $W'$ along $\Gamma$ and $\Gamma'$.
 \end{theorem}
\begin{remark}According to \cite{JTZ},  Heegaard Floer homology groups depend on the choice of basepoint. Similarly,  the cobordism-induced maps depend on the path connecting them  \cite{ZemkeGraphCob}.  Despite this, we  suppress this data from the notation.
\end{remark}
\begin{proof} Pick a handle decomposition $\mathcal{H}$ of $W$ relative to $Y_1$, and adapted to $\Gamma$ in the following sense: there is a point $w\in Y_1$ in the complement of the attaching regions for all the handles of  $\mathcal{H}$, such that $\Gamma$ is the properly embedded arc from $Y_1$ to $Y_2$ obtained as the trace of $w$. Such a decomposition can be obtained from a generic Morse function with gradient vector field for which $\Gamma$ is a flowline.  Similarly, let $\mathcal{H}'$ be a handle decomposition of $W'$ adapted to $\Gamma'$.  Since $\Gamma$ and $\Gamma'$ are in the complement of the attaching regions for all of the handles of $W$ and $W'$, there is a handle decomposition of $W\otimes W'$ given by adding handles, in turn, to either $W\setminus \nu(\Gamma)$ or $W'\setminus \nu(\Gamma')$. Thus, it suffices to prove the statement in the special cases where $W$ is a 1, 2, or 3-handle addition and $W'$ is the product cobordism $Y_1'\times I$, endowed with the canonical extension of $\spinc_1'$ (whose associated map is the identity). 

First, suppose $W$ is a 1-handle addition and, therefore, a cobordism from $Y_1$ to $Y_1\#(S^1\times S^2)$. Ozsv\'ath and Szab\'o define the map induced by $W$ as follows: 
there is a unique $\SpinC$ structure $\spinct$ on $W$ extending $\spinc_1\in \SpinC(Y_1)$ which restricts to $Y_1\#(S^1\times S^2)$ as $\spinc_1\#\spinc_0$ where $\spinc_0$ is the unique $\SpinC$-structure on $S^1\times S^2$ with $c_1(\spinc_0)=0$. Given a Heegaard diagram $(\Sigma,\alphas,\betas,w)$ for $Y_1$ and  the standard (weakly admissible) genus one Heegaard diagram for $S^1\times S^2$ with two generators, $(E,\alpha,\beta,w_0)$, there is a Heegaard diagram $(\Sigma\# E,\alphas\cup \alpha,\betas\cup\beta, w)$ for $Y_1\# (S^1\times S^2)$. The map $F_{W,\spinct}$ induced by $W$ is defined by a chain map  which, for $\x\in \CFa(Y_1,\spinc_1)$, is given by  $f_{W, \spinct}(\x)=\x\otimes\theta_+$.  Here, $\theta_{+}$ is the element of higher  relative Maslov grading in $(E,\alpha,\beta,w_0)$. At the same time, the chain map $f_{W\otimes W',\spinct\#\spinct'}$ is defined  by sending $\x\otimes \y$ in $\CFa(Y_1\#Y_1',\spinc_1\#\spinc_1') $ to $\x\otimes\theta_+\otimes \y$ in $\CFa(Y_1\#(S^1\times S^2)\#Y_1', \spinc_1\#\spinc_0\#\spinc_1')$. Here, we are using quasi-isomorphisms provided by the K{\"u}nneth theorem \cite[Theorem 1.5]{HolDiskTwo}.  Considering the induced maps on homology, we have: $F_{W, \spinct}\otimes \Id= F_{W\otimes W',\spinct\#\spinct'}$.  The case of a 3-handle addition is formally the same, since the maps in that case are dual to the 1-handle maps. See Section 4.3 of \cite{HolDiskFour} for more details. 

Next, we consider the case where $W$ is a 2-handle addition. In this case, $Y_2$ is given by integral surgery along a framed knot $K\subset Y_1$, and the map induced on Floer homology by $W$ is defined by counting pseudo-holomorphic triangles associated to an adapted Heegaard triple diagram.  To describe this, let $(\Sigma,\alphas,\betas,w)$ be a Heegaard diagram for $Y_1$ where the final $\beta$-curve is the meridian of the framed knot $K$. Then we have a related Heegaard diagram $(\Sigma,\alphas,\gammas,w)$ for $Y_2$ where the first $(g-1)$ $\gamma$-curves are small Hamiltonian translates of the first $(g-1)$ $\beta$-curves and $\gamma_g$ is the longitude for $K$ corresponding to the 2-handle addition. Together, this data yields a Heegaard triple diagram $(\Sigma,\alphas,\betas,\gammas,w)$ specifying a 4-manifold $X_{\alpha\beta\gamma}$ with $\partial X_{\alpha\beta\gamma}=-Y_{\alpha\beta}-Y_{\beta\gamma}+Y_{\alpha\gamma},$ where $Y_{\alpha\beta}=Y_1$, $Y_{\beta\gamma}=\#^{g-1}S^1\times S^2$ and $Y_{\alpha\gamma}=Y_2$. For further details about the construction of $X_{\alpha\beta\gamma}$ see  \cite[Section 4.1]{HolDiskFour}. Observe that $W$ can be recovered from $X_{\alpha\beta\gamma}$ by capping off  the $\#^{g-1}S^1\times S^2$ boundary component with $\large\natural^{g-1}S^1\times B^3$. \ons\  associate  a chain map  to $W$ by $$f_{W,\spinct}(\x):=f_{{\alpha\beta\gamma}}(\x\otimes\Theta_{top}),$$
where the latter is a sum, over all $\y$ generating  $\CFa(Y_2)$, of the number of pseudo-holomorphic triangles in Sym$^g(\Sigma)\setminus V_w$ whose homotopy class represents $\spinct$ and  whose vertices map to $\x,\Theta_{top}$ and $\y$ (here, and throughout the proof, we conflate the homology class $\Theta_{top}$ with its unique chain representative on the given Heegaard diagram).  The map induced on homology is denoted $F_{W,\spinct}$.

The map induced by $W\otimes W'$ on Floer homology admits a similar description. Given  a Heegaard diagram $(\Sigma',\alphas',\betas',w')$ for $Y_1'$, we construct  a Heegaard triple diagram $(\Sigma',\alphas',\betas',\gammas',w')$, where the curves $\gammas'$ are small Hamiltonian translates of the curves $\betas'$. We then form the connected sum of this latter Heegaard triple diagram with the one associated to the 2-handle cobordism above: \[(\Sigma,\alphas,\betas,\gammas,w)\underset{{w=w'}}\#(\Sigma',\alphas',\betas',\gammas',w').\] 
\noindent That is, we form the connected sum of $\Sigma$ with $\Sigma'$ along  neighborhoods of the basepoints $w$ and $w'$, and let the curves from the constituent diagrams descend to  $\Sigma\#\Sigma'$.  The basepoints naturally descend to a basepoint $w$, living in the region of the triple diagram corresponding to  the regions containing the basepoints.
This  triple diagram  describes a 4-manifold  
 with boundary components $$Y_{\alpha\cup\alpha',\beta\cup\beta'}=Y_1\#Y_1',$$ $$Y_{\beta\cup\beta',\gamma\cup\gamma'}=\#^{g-1+g'}(S^1\times S^2),$$ and $$Y_{\alpha\cup\alpha',\gamma\cup\gamma'}=Y_2\#Y_1',$$ where $g$ is the genus of $\Sigma$ and $g'$ is the genus of $\Sigma'$. A chain map induced by $W\otimes W'$ is defined by $$f_{W\otimes W',\spinct\#\spinct'}(\x\otimes \y):=f_{{\alpha\cup\alpha',\beta\cup\beta',\gamma\cup\gamma'}}((\x\otimes \y)\otimes \Theta_{top}).$$

\noindent Here, $\Theta_{top}$ is the top graded generator of $\CFa(\#^{g-1+g'}(S^1\times S^2))$ coming from the Heegaard diagram $(\Sigma\#\Sigma',\betas\cup\betas', \gammas\cup\gammas',w)$. This generator decomposes as $\Theta_{top}=\Theta^{g-1}_{top}\otimes \Theta^{g'}_{top}$ where $\Theta^{g-1}_{top}$ is the top graded generator for $(\Sigma, \betas,\gammas,w)$  and $\Theta^{g'}_{top}$ is the top graded generator in $(\Sigma, \betas',\gammas',w)$. 

 Like the chain complexes associated to the connected sum of Heegaard diagrams, the chain map $f_{{\alpha\cup\alpha',\beta\cup\beta',\gamma\cup\gamma'}}$ splits  as a tensor product:
$$f_{{\alpha\cup\alpha',\beta\cup\beta',\gamma\cup\gamma'}}((\x\otimes \y)\otimes \Theta_{top})= f_{{\alpha,\beta,\gamma}}(\x\otimes\Theta^{g-1}_{top})\otimes f_{{\alpha',\beta',\gamma'}}(\y\otimes \Theta^{g'}_{top})$$ where $f_{\alpha',\beta',\gamma'}$ is the chain map associated to the Heegaard triple $(\Sigma',\alphas',\betas',\gammas', z')$. This splitting, as with the K{\"u}nneth theorem for the hat Floer homology of a connected sum, can be easily proved by appealing to the ``localization principle" for holomorphic triangles whose domains split as a disjoint union (see \cite[Section 9.4]{RasThesis}).   Since the connected sum of diagrams is performed near the basepoint $w$, and the hat theory prohibits the domains of disks and triangles from entering this region, all moduli spaces  split as a cartesian product of  moduli spaces  associated to the two Heegaard triple diagrams.    Finally, since the $\gammas'$ curves are translates of the $\betas'$ curves, \[f_{{\alpha',\beta',\gamma'}}(\y\otimes \Theta^{g'}_{top})=\tilde{\y}+\mathrm{lower\ order\ terms\ with\ respect\ to\ symplectic\ area},\]  where $\tilde{\y}$ is the generator associated to the ``closest" point map. It follows that the map on homology can be taken to be the identity; see \cite[Section 9]{HolDisk}  for more details on the symplectic area filtration, specifically the discussion starting on pg. 1122 of op. cit. 
\end{proof}
\begin{remark} The product formula extends to the other versions of Floer homology, either by using a more sophisticated degeneration and gluing argument for holomorphic triangles, or by invoking an argument similar to the one in \cite[Section 4]{AbsGrad}.  Another proof can  be obtained using Zemke's graph cobordism TQFT \cite{ZemkeGraphCob}.  In that context, one considers the 3-handle cobordism from $Y_1\#Y_1'$ to $Y_1\sqcup Y_1'$, composed with $W\sqcup W'$, composed with the 1-handle cobordism from $Y_2\#Y_2'$ to $Y_2\sqcup Y_2'$.  The resulting graph cobordism (where the graphs in the 1- and 3-handles are the obvious trivalent graphs with 3-edges) satisfies the product formula given.  One can surger this cobordism along a neighborhood of the cycle arising from $\Gamma\sqcup \Gamma'$ joined to the vertices in the 1- and 3-handles.  This results in the arc sum, equipped with the given path, and the resulting maps are easily argued to agree.  We opted for the  proof given, as it is  elementary and self-contained.
\end{remark}

\subsection{Vanishing of maps on Floer homology}
 In this subsection we prove the  vanishing result (Theorem \ref{thm:cobordismbound}) which is central to our proof of the relative adjunction inequality.  

Let $\Sigma$ be a closed surface embedded in a cobordism $W$ whose incoming end is a $3$-manifold $Y$, and  let  $\gamma$  be a properly embedded arc connecting $Y$ to $\Sigma$. Let $N=N(Y\cup \gamma\cup\Sigma)$ be a regular neighborhood of $Y\cup\gamma\cup \Sigma$. Then $$\partial N=-Y\sqcup (Y\#\Cbundle),$$ where $\Cbundle$ is the circle bundle over $\Sigma$ with Euler number $[\Sigma]^2$. 
\begin{figure}
\begin{subfigure}{.5\textwidth}
\centering
\includegraphics[height= 1.5in]{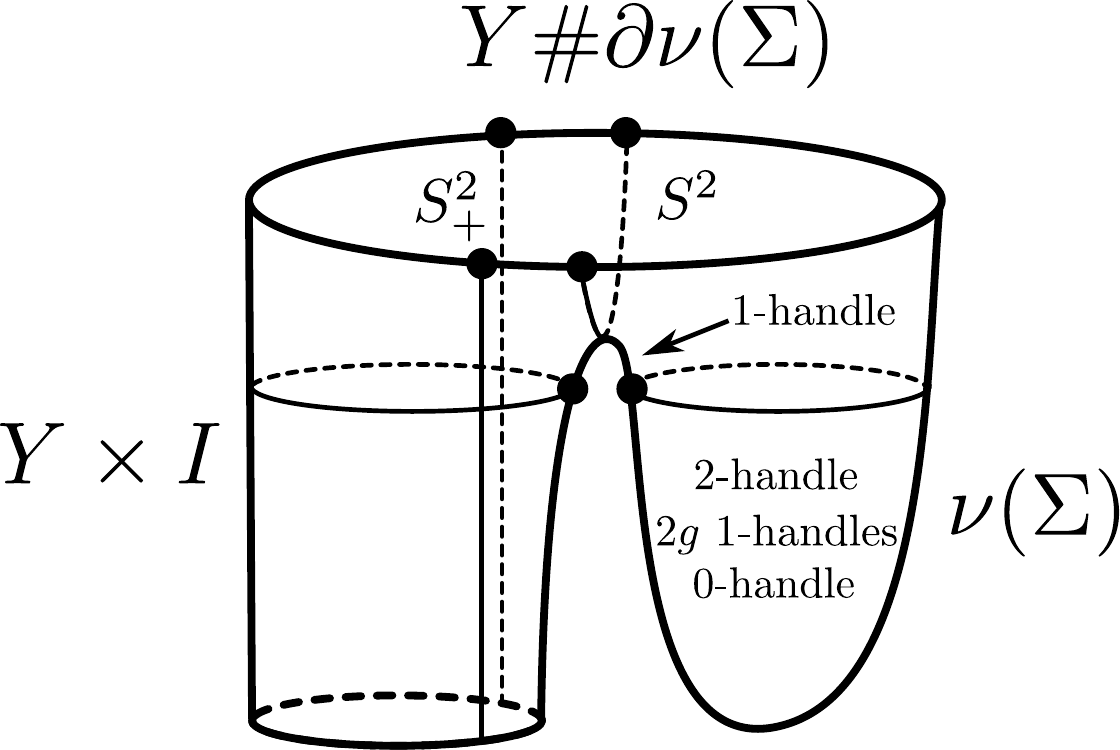}
\caption{\centering A schematic of $(Y\times I)\bignatural \nu(\Sigma)$.}
\label{Morse_1}
\end{subfigure}

\begin{subfigure}{.45\textwidth}
\centering
\includegraphics[height=1.4 in]{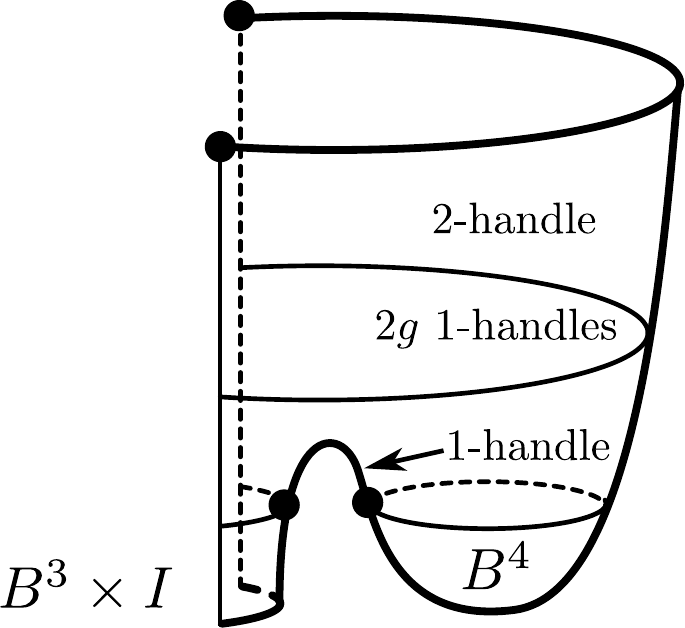}
\caption{\centering A schematic of $(B^3\times I)\bignatural \nu(\Sigma)$.}
\label{Morse_2}
\end{subfigure}
\begin{subfigure}{.45\textwidth}
\centering
\includegraphics[height=1.4 in]{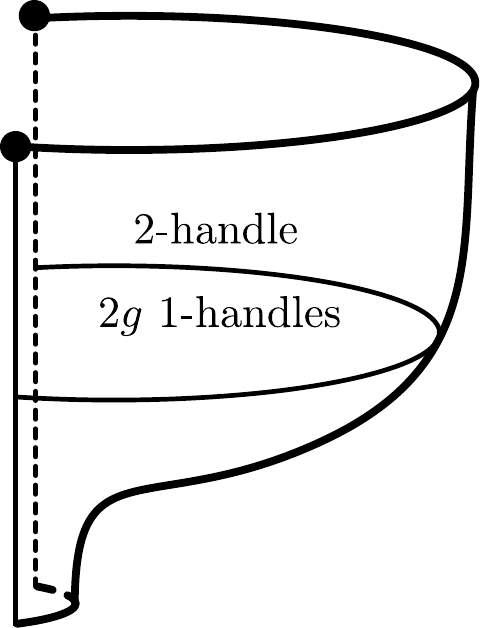}
\caption{\centering After handle cancellation, $(\nu(\Sigma)-B^4)\setminus(B^3\times I)$.}
\label{Morse_3}
\end{subfigure}
\caption{}
\end{figure}   

\begin{lemma}\label{lemma:arcsumdecomposition} The 4-manifold $N$ described above is diffeomorphic to the boundary connected sum $$(Y\times[0,1] )\bignatural \Dbundle$$ where $\Dbundle$ is the disk bundle over $\hatSigma$ with Euler number $[\hatSigma]^2$. Alternatively, $N$ can be smoothly decomposed as an arc sum $$ (Y\times I)\setminus (B^3\times I) \bigcup_{S^2\times I} (\Dbundle-B^4)\setminus (B^3\times I).$$
\end{lemma}

\begin{proof} Given handle descriptions for the disjoint manifolds $Y\times I$ and $\Dbundle$ arising as neighborhoods of $Y$ and $\Sigma$, respectively, a handle description for $(Y\times I )\bignatural \Dbundle$ is given by attaching a 4-dimensional 1-handle  to connect them.  This 1-handle can be identified with the part of the neighborhood of $\gamma$ outside the neighborhoods of $Y$ and $\Sigma$, verifying the first claim.  Note that the resulting  handle description corresponds to a Morse function on $(Y\times I )\bignatural \Dbundle$ where the index 1 critical point corresponding to the connecting 1-handle has largest critical value. See Figure \ref{Morse_1}.

As illustrated by Figure \ref{Morse_1}, the belt sphere $S^2$ of the 1-handle separates the ``upper" boundary $Y\#\partial \nu(\Sigma)$ into its (punctured) summands, $Y-B^3$ and $\nu(\Sigma)-B^3$.  The image of a boundary parallel sphere $S^2_+$ in $Y-B^3$ under the  
 downward  gradient flow of the Morse function is a properly embedded $S^2_+\times I$. Removing this separates $N$ into two pieces $(Y\times I )- (B^3\times I)$ and $(B^3\times I) \bignatural \Dbundle$. It remains to show that $(B^3\times I)\bignatural \Dbundle\cong (\Dbundle-B^4)\setminus(B^3\times I)$. 

To this end, change the Morse function on $(B^3\times I)\bignatural \Dbundle$ so that handles are added in order of index. Specifically, begin with $B^3\times I$ and $B^4$ and add a 1-handle to connect them. Then to the boundary of $B^4$ attach the remaining 1- and 2-handles of $\Dbundle$. 

Now $B^4$ cancels  the connecting 1-handle, so this manifold is diffeomorphic to one with a handle decomposition built from $B^3\times I$ by attaching 1- and 2-handles along $B^3\times \{1\}$. But, $B^3\times I$ union the 1- and 2-handles is easily identified with  $(\Dbundle-B^4)\setminus(B^3\times I)$. See Figures \ref{Morse_2} and \ref{Morse_3} for a schematic.
\end{proof}

We are now ready to give a proof of Lemma 3.5 of \cite{FourBall}. Our statement and proof differ from (and correct) the one given there.  See the remark below the proof. \begin{lemma}\label{lemma:circlebundle}
Let $\Dbundle$ be the disk bundle over a closed oriented connected surface $\Sigma$ of genus $g=g(\Sigma)$. The map $$F_{\Dbundle-B^4,\spinct}:\HFhat(S^3)\to\HFhat(\Cbundle,\spinct|_{\Cbundle})$$ is trivial for all $\spinct\in\SpinC(\Dbundle-B^4)$ such that $$\langle c_1(\spinct),[\hatSigma]\rangle+[\Sigma]^2>2g(\Sigma).$$ 
\end{lemma}
\begin{proof} 
The disk bundle $\Dbundle$ has a handle decomposition with a single 0-handle, $2g$ 1-handles and a single 2-handle. This  decomposition is described explicitly via a handlebody diagram obtained from a diagram for $\bignatural^{2g} S^1\times B^3$ by attaching a 2-handle along the Borromean knot $B_g$ (see \cite[Figure 12.5]{GS} or \cite[Figure 16]{Knots} for pictures of $B_3$ and $B_1$, respectively). 

Thus, $\Dbundle-B^4=W_1\cup_{\#^{2g}\SoneStwo}W_2$, where $W_1=(\bignatural^{2g}S^1\times B^3)-B^4$ and $W_2$ is the cobordism associated to the 2-handle addition along $B_g$. By Lemma \ref{lemma:splitting}, $\spinct=\spinct_1\#\spinct_2$ and the map $F_{\Dbundle-B^4,\spinct}$ factors as $$\HFhat(S^3)\xrightarrow{G_{W_1,\spinct_1}} \HFhat(\#^{2g}S^1\times S^2)\xrightarrow{G_{W_2,\spinct_2}} \HFhat(\Cbundle,\spinct|_{\Cbundle}).$$

Consider the map $G_{W_2,\spinct_2}$. In Section 9 of \cite{Knots}, Ozsv\'ath and Szab\'o calculate $CFK^\infty(B_g)$. There they show that in Alexander grading $k$, $$\widehat{HFK}(\#^{2g}S^1\times S^2,B_g,k)=\Lambda^{g+k}H^1(\hatSigma),$$ supported in Maslov grading $k$. Moreover, they show $$CFK^\infty(B_g)=\widehat{HFK}(\#^{2g}S^1\times S^2,B_g)\otimes_{\F_2} \F_2[U,U^{-1}].$$ Theorem 4.1 of \cite{Knots}, also known as the Large Surgery Theorem, implies that if the framing of the 2-handle (which equals the Euler number of the disk bundle) is negative and less than or equal to $-2g+1$ then the map $$G_{W_2,\spinct_2}:\HFhat(\#^{2g}S^1\times S^2)\to \HFhat(\Cbundle), \spinct|_{\Cbundle})$$ can be calculated from the map $$CFK^\infty(B_g)\{i=0\}\to CFK^\infty(B_g)\{\min(i,j-k)=0\},$$ which is a composition of a quotient followed by an inclusion. Here we have enumerated $\SpinC$-structures on $\#^{2g}S^1\times S^2$ that extend over the 2-handle addition so that $\langle c_1(\spinct_2),[\hatSigma]\rangle +[\hatSigma]^2=2k$. If $k>g$, then $G_{W_2,\spinct_2}$ is trivial, since the generators of $CFK^\infty(B_g)\{i=0\}$ have Alexander grading less than or equal to $g$. Finally, observing that $$\langle c_1(\spinct),[\hatSigma]\rangle +[\hatSigma]^2=\langle c_1(\spinct_2),[\hatSigma]\rangle +[\hatSigma]^2, $$ the result follows in the special case that the Euler number of the disk bundle is less than or equal to $-2g+1$.  Note that \cite[Theorem 4.1]{Knots} only states that the surgery formula holds provided that the framing is sufficiently negative.  That $-2g+1$ is negative enough follows from the argument discussed in \cite[Remark 4.3]{Knots}, applied in the context of the integer surgeries exact sequence with negative framings, \cite[Remark 9.20]{HolDiskTwo}. 

 The result for general Euler number follows from this special case using the blow-up formula.  Indeed, assume there is a $\SpinC$ structure $\spinct$ on the punctured Euler number $n$ disk bundle with $F_{{\Dbundle}- B^4,\spinct}\ne 0$ and which satisfies $\langle c_1(\spinct),[\hatSigma]\rangle+[\Sigma]^2>2g(\Sigma).$ Then we can blow up the disk bundle $p$ times, so that $n-p\le -2g+1$.  The blow-up formula  \cite[Theorem 3.7]{HolDiskFour} indicates that on the blown-up disk bundle $\widehat{\Dbundle}=\Dbundle\#^p\CPbar$ there is a $\SpinC$ structure $\hat{\spinct}$ satisfying 
\begin{itemize}
\item $\langle c_1( \hat{\spinct} ),[{\Sigma}]\rangle=\langle c_1({\spinct}),[{\Sigma}]\rangle$
\item $\langle c_1(\hat{\spinct}), [E_i]\rangle =1$, for the class of each exceptional sphere $E_i$, $i=1,\ldots, p$.
\item $F_{\widehat{\Dbundle}- B^4,\widehat{\spinct}}=F_{{\Dbundle}- B^4,{\spinct}}\ne 0$.
\end{itemize}
Tubing $\Sigma$ to each of the exceptional spheres produces another genus $g$ surface    $\widehat{\Sigma}$ whose  homology class is $[\widehat{\Sigma}]=[\Sigma]+[E_1]+\ldots+[E_p]$. Noting that $[E_i]\cdot [E_j]=0$ if $i\ne j$ and $-1$ if $i=j$, it follows that the self-intersection of $\widehat{\Sigma}$ equals $n-p\le -2g+1$, and we can apply the previous case to its neighborhood.  But \[\langle c_1(\hat{\spinct}),[\widehat{\Sigma}]\rangle+[\widehat{\Sigma}]^2=\langle c_1(\spinct),[{\Sigma}]\rangle+[\Sigma]^2>2g(\Sigma),\]
and therefore  $F_{\nu(\widehat{\Sigma})- B^4,\widehat{\spinct}}=0$.  But the cobordism map for $\widehat{\nu(\Sigma)}-B^4$ (the punctured blown-up disk bundle) factors through the map associated to ${\nu(\widehat{\Sigma})}-B^4$  (the punctured neighborhood  of $\widehat{\Sigma}$), hence must also be zero, a contradiction.

\end{proof}

\begin{remark}\label{circle-error} Lemma 3.5 of \cite{FourBall} states that the map on Floer homology induced by the punctured disk bundle vanishes whenever
$$\langle c_1(\spinct),[\hatSigma]\rangle+[\Sigma]^2\ge 2g(\Sigma).$$ 
Examination of our proof shows that when the Euler number is sufficiently negative the map is actually non-trivial for the $\SpinC$ structure satisfying $\langle c_1(\spinct),[\hatSigma]\rangle+[\Sigma]^2= 2g(\Sigma)$.  Indeed, the map $G_{W_1}$ associated to the $1$-handles has image $\Theta_{top}$.  But this latter class lives in Alexander grading $g$ in the filtration of $\CFa(\#^{2g}S^1\times S^2)$ associated to $B_g$.   Hence it survives in the quotient and inclusion to $CFK^\infty(B_g)\{\min(i,j-g)=0\}$.  The corrected vanishing result, when traced through the arguments of \cite{FourBall}, leads to the following  $4$-genus bound for $\tau$, which is weaker than the bound asserted in op. cit.:
\[  \tau(K)\le g_4(K)+1.\]
We will establish the asserted bound  $\tau(K)\le g_4(K)$ used throughout the literature by exploiting the product formula for arc sums of cobordisms  in conjunction with the additivity of $\tau$ invariants under connected sum.
\end{remark}

Together with the product formula for arc sums, the  previous two lemmas  yield the following vanishing result, which will play a key role in the proof of Theorem \ref{thm:relativeadjunction}.

\begin{theorem}[Vanishing Theorem]\label{thm:cobordismbound} Let $\hatSigma$ be a closed, oriented, surface,  smoothly embedded in a 4-manifold $W$ such that $\partial W=-Y\sqcup Y'$. Then $$F_{W,\spinct}:\HFhat(Y, \spinct|_{Y})\to \HFhat(Y', \spinct|_{Y'})$$ is the zero map for all $\spinct$ satisfying $\langle c_1(\spinct),[\hatSigma]\rangle +[\hatSigma]^2>2g(\hatSigma).$

\end{theorem}
\begin{proof} Let $N=N(Y\cup\gamma\cup\hatSigma)$ be a regular neighborhood of $Y$, the surface, and an arc connecting them, and 
write $W$ as $N\cup_{\partial N} W'$ where $W'$ is the complement of $N$. We have identifications (coming from, say, the Mayer-Vietoris sequence) $$H^1(N)\cong H^1(Y)\oplus H^1(\Dbundle)$$ and $$H^1(\partial N)\cong H^1(Y)\oplus H^1(\Cbundle),$$ which are natural with respect to the restriction maps.
Since the restriction map  $H^1(\Dbundle)\to H^1(\Cbundle)$ is surjective, the map $H^1(N)\to H^1(\partial N)$ is also surjective. Lemma \ref{lemma:splitting} then implies that $\spinct=\spinct_1\#\spinct_2$ where $\spinct_1=\spinct|_N$ and $\spinct_2=\spinct|_{W'}$. The composition law for cobordism maps \cite[Theorem 3.4]{HolDiskFour} shows that $F_{W,\spinct}$ factors as $F_{W',\spinct_2}\circ F_{N,\spinct_1}$, and $$\langle c_1(\spinct),[\hatSigma]\rangle +[\hatSigma]^2=\langle c_1(\spinct_1),[\hatSigma]\rangle +[\hatSigma]^2,$$ as the surface is contained in $N$.  It therefore suffices to show that $F_{N,\spinct_1}$ vanishes whenever $\langle c_1(\spinct_1),[\hatSigma]\rangle +[\hatSigma]^2>2g(\hatSigma)$.

Lemma \ref{lemma:arcsumdecomposition} implies that $N$ smoothly decomposes as an arc sum $$(Y\times I)\setminus (B^3\times I) \bigcup_{S^2\times I} (\Dbundle-B^4)\setminus (B^3\times I).$$ Applying Lemma \ref{lemma:splitting} to this decomposition, $\spinct_1=\spincu\#\spincu'$ where $\spincu=\spinct_1|_{Y\times I}$ and $\spincu'=\spinct_1|_{\Dbundle}$. Furthermore, 
$\langle c_1(\spinct_1),[\Sigma]\rangle+[\Sigma]^2=\langle c_1(\spincu'),[\Sigma]\rangle+[\Sigma]^2.$

Now, suppose $\langle c_1(\spincu'),[\hatSigma]\rangle+[\hatSigma]^2>2g(\hatSigma)$ and consider the commutative diagram given by the product formula, Theorem \ref{thm:Kun}:
$$
\begin{CD}
	\HFa(Y,\spincu|_{Y})\otimes\HFa(S^3) 
	@>{   {  F_{Y\times I,\spincu}  } \otimes {F_{  \Dbundle - B^4,\spincu' } }}>>                  
	\HFa(Y,\spincu|_{Y})\otimes\HFa(\Cbundle,\spincu'|_{\Cbundle}) \\
	@V{\cong}VV    \   @V{\cong}VV \\
		\HFa(Y,\spincu|_{Y}) @>{\ \ \ \ \ \ \ \  F_{N,{\spincu\#\spincu'}}\ \ \ \ \ \ }>> \HFa(Y\#\Cbundle, \spincu\#\spincu'|_{Y\#\Cbundle})\\
\end{CD}
$$
Lemma \ref{lemma:circlebundle} now implies that $F_{\Dbundle - B^4,\spincu'}$ is trivial.  Thus, $F_{N,\spincu\#\spincu'}$ is also trivial.
\end{proof}
\bigskip

The remainder of this section is aimed at specifying the  manner in which  $\tau$ invariants constrain the 2-handle cobordism maps, constraints laid out in  Proposition \ref{prop:FourDInterp}.   To make this precise, it will be helpful to establish some numerology derived from the algebraic topology of a handle attachment along a rationally null-homologous knot.  The next two subsections accomplish this, with the final subsection proving the key
proposition.

\subsection{Framings for rationally null-homologous knots} Regardless of its homology class, a knot $K$ has a well-defined meridian $\mu$ which is given by the isotopy class of the boundary of a  disk intersecting $K$ in a single point. A framing for $K$ is equivalent to a choice of curve $\lambda$ in $\partial\nu(K)$ so that the pair $([\mu],[\lambda])$ forms a basis for $H_1(\partial\nu(K))\cong \Z\oplus \Z$. Given an initial choice of $\lambda$, every other choice of framing is given, on the level of homology, by $\lambda+n\mu$ for some $n\in\Z$.

Let $K\subset Y$ be a knot whose homology class has order $q$. In the long exact sequence of the pair $(Y-\nu(K),\partial\nu(K))$, the kernel of  $i_*:H_1(\partial\nu(K))\to H_1(Y-\nu(K))$  is isomorphic to $\Z$ and is generated by the homology class of $S\cap \partial \nu(K)$ where $S$ is a rational Seifert surface for $K$. Returning to the preceding paragraph, for an initial choice of $\lambda$, we can write $S\cap \partial \nu(K)=q\lambda+r\mu$ and the choices of $\lambda$ are in bijection with representatives of the congruence class of $r$ modulo $q$. 

\begin{defn}\label{defn:longitude}
The \emph{canonical longitude}, $\lambda_{\can}$, of $K$ is the unique choice of framing  such that $S\cap \partial \nu(K)=q\lambda_{\can}+r\mu$ with  $0\leq r<q$.
\end{defn}
\noindent Equivalently, for any choice of $\lambda$ the fraction $\frac{r}{q}$, viewed in $\Q\slash \Z$, is the self-pairing of $[K]$ under the linking form on $H_1(Y)$, and $\lambda_{\can}$ is the unique choice of longitude so that $\frac{r}{q}\in\Q$ is the coset representative of the self-pairing lying in the interval $[0,1)$. If $K$ is null-homologous, then $q=1$, $r=0$ and $\lambda_\can$ is the usual Seifert framing.

\subsection{Integer surgery and surgery cobordisms}
For a knot $K$ in a 3-manifold $Y$, let $\Surg$ be the 3-manifold obtained by removing $\nu(K)$ from $Y$ and filling $Y-\nu(K)$ along the $n$-framed longitude, $\lambda_{\can}+n\mu$. Attaching a 4-dimensional 2-handle to $Y\times \{1\}\subset Y\times [0,1]$ along $K\times \{1\}$ with framing $\framing$ determines a cobordism $\Wsurg$ from $Y$ to $\Surg$. As an oriented manifold, $\Wsurg$ has boundary $-Y\sqcup \Surg$.

Two variations of this cobordism interest us here: $\Wsurgneg$ and $\Wsurgdual$. The manifold $\Wsurgneg$ is the cobordism described above from $Y$ to $\Surgneg$ where we assume  $-\framing<0$. On the other hand, $\Wsurgdual$ is $\Wsurg$  with its orientation reversed and viewed as a cobordism in the other direction so that $\Wsurgdual$ has boundary $-\Surg\sqcup Y$ and, viewed as a cobordism, it goes from $\Surg$ to $Y$. Both $\Wsurgneg$ and $\Wsurgdual$ are negative definite for $\framing>0$.

The surgery cobordisms $\Wsurgneg$ and $\Wsurgdual$ induce maps on Floer homology: $$ F_{\Wsurgneg, \spinct}: \HFhat(Y, \spinct|_Y)\to\HFhat(\Surgneg,\spinct|_{\Surgneg})$$ and $$ F_{\Wsurgdual, \spincr}: \HFhat(\Surg, \spincr|_{\Surg})\to\HFhat(Y,\spincr|_{Y})$$ for each $\spinct\in\SpinC(\Wsurgneg)$ and $\spincr\in\SpinC(\Wsurgdual)$. It will be useful to enumerate  these maps.

 To this end, observe that the set of extensions of a fixed $\SpinC$ structure $\spinc\in \SpinC(Y)$ over $\Wsurgneg$ is in affine bijection with classes in $H^2(\Wsurgneg,Y)\cong\Z$.   We  will establish a preferred bijection using (rational) Chern class evaluations.  To do this, first note that the Alexander gradings of the lifts $G^{-1}_{Y,K}(\spinc)$ under the filling map  $G_{Y,K}:\SpinC(Y,K)\rightarrow \SpinC(Y)$ form a coset in $\Q\slash\Z$, denoted $A_{Y,K,[S]}(\spinc)$. Let $k_\spinc$ denote the coset representative in the interval $(-\frac{1}{2},\frac{1}{2}]$.  In these terms, we let \[\Mapneg:=F_{\Wsurgneg, \spinct_m^\spinc}\] denote the map induced on Floer homology by  $\Wsurgneg$, equipped with the unique $\SpinC$ structure $\spinct^\spinc_m$ for which $\spinct^\spinc_m|_Y=\spinc$ and 
\begin{equation}\label{enumeration}
\langle c_1(\spinct_m^\spinc),[\SSurfD]\rangle +[\SSurfD]^2 = 2(k_{\spinc}+m),
\end{equation}
\noindent where $[\SSurfD]\in H_2(\Wsurgneg;\Q)$ is the homology class represented by the core of the 2-handle, ``capped-off" with the rational Seifert surface.  To describe this class, let $D$ denote the core disk of the  2-handle, whose class in $H_2(\Wsurgneg,Y)$ is the generator with $\partial D=-K$. Then $[\SSurfD]$ is the lift of $[D]$ (regarded as a rational class) to $H_2(\Wsurgneg;\Q)$ represented by $D$ capped off with the rational 2-chain  $\frac{1}{\order} \SSurf$, where $\SSurf$ is a rational Seifert surface; that is, $[\SSurfD]=[\frac{1}{\order} \SSurf+  D]$. Strictly speaking, to interpret $\frac{1}{\order} \SSurf+  D$ as a 2-cycle in $C_2(\Wsurgneg;\Q)$ we must pick a homology in $C_2(K;\Q)$ between the rational 1-cycles $\frac{1}{\order} \partial\SSurf$ and $\partial D$, but  the ambiguity introduced by this choice lives in $H_2(K)=0$. 

Similarly, define \[\Mappos:=F_{\Wsurgdual, \spincr_m^\spinc},\] where  $\spincr_m^\spinc$ is the unique $\SpinC$-structure on $\Wsurgdual$ such that $\spincr_m^\spinc|_Y=\spinc$ and 
\begin{equation}
\langle c_1(\spincr_m^\spinc),[\SSurfD]\rangle -[\SSurfD]^2 = 2(k_{\spinc}+m).
\end{equation}
Again, $[D]$ is the generator of $H_2(\Wsurgdual,Y)$ with $\partial D=-K$ and $[\SSurfD]$ denotes the lift of $[D]$ to $H_2(\Wsurgdual;\Q)$ associated to the rational Seifert surface $S$ for $K$.

\subsection{The $\tau$ invariant from a 4-dimensional perspective}
For knots in the 3-sphere, the $\tau$ invariant indicates a threshold in the enumeration of $\SpinC$ structures before which the cobordism maps $\Mapneg$ mentioned above must be nontrivial \cite{FourBall}. For  rationally null-homologous knots $K\subset Y$, analogous results hold for both $\Mapneg$ and $\Mappos$.

\begin{prop}\label{prop:FourDInterp}
Let $\alpha$ be a nontrivial element of $\widehat{HF}(Y,\spinc)$. For $n$ positive and sufficiently large, we have the following:
\begin{itemize}
\item if $m> \tau_{\alpha}(Y,K)-k_\spinc$ then $\alpha\in\Image(\Mappos)$;
\item if $m<\tau_{\alpha}(Y,K)-k_\spinc$ then  $\alpha\notin\Image(\Mappos)$.
\\
\item if $m< \tau_{\alpha}(Y,K)-k_\spinc$ then $\Mapneg(\alpha)\neq 0$;
\item if $m>\tau_{\alpha}(Y,K)-k_\spinc$ then  $\Mapneg(\alpha)= 0$.
\end{itemize} 
Here, as above, $k_\spinc$ denotes the unique element in $(-\frac{1}{2},\frac{1}{2}]$ arising as an Alexander grading of a relative $\SpinC$ structure in $G_{Y,K}^{-1}(\spinc)$.
\end{prop}
\begin{proof} 
The proof relies on the ``Large Surgery Theorem" for rationally null-homologous knots; see \cite[Theorem 4.1]{RationalSurgeries} and \cite[Theorem 5.8]{Hedden-Levine-surgery} for the case of positive surgeries and \cite[Theorem 4.2]{Raoux} for the statement for negative surgeries. 

Let $\sC_\spinc$ denote the complex $\CFKinf(Y,[S],K,\spinc)$ and assume $n$ is large enough so that the Large Surgery Theorem holds for $n$-surgery as well as $-n$-surgery along $K$. 

For $n$-surgery, the Large Surgery Theorem implies that the map $$\Mappos:\HFhat(\Surg,\spinct^\spinc_m|_{\Surg})\to \HFhat(Y,\spinc)$$ can be identified with the map induced on homology by $$f_m:\sC_\spinc\{\max(i,j-m)=0\}\to \sC_\spinc\{i=0\}$$
where $f_m=\iota_{m}\circ q_m$ is the composition of the quotient map $$q_m:\sC_\spinc\{\max(i,j-m)=0\}\to\sC_\spinc\{i=0, j\leq m\}$$ followed by the inclusion $$\iota_{m}:\sC_\spinc\{i=0, j\leq m\}=\Filt_{m}(Y,[S],K)\hookrightarrow \sC_\spinc\{i=0\}=\CFa(Y,\spinc).$$ 
 
 Now observe that if $m<\tau_\alpha(Y,[S],K)-k_\spinc$, then $\alpha$ is not in the image of $I_{m}$ and hence not in the image of $\Mappos$.

On the other hand, $\sC_\spinc\{i=0, j\le m-1\}$ naturally includes into the complex $\sC_\spinc\{\max(i,j-m)=0\}$. This gives a factorization of the map $f_m$ through
$$\iota_{m-1}:\sC_\spinc\{i=0, j\leq m-1\}\to \sC_\spinc\{i=0\}.$$ 
If $m>\tau_\alpha(Y,[S],K)-k_\spinc$ then $\alpha$ is in the image of $I_{m-1}$ and is thus also  in the image of $\Mappos$.

The argument for $-n$-surgery is similar and is the same as the one given in \cite[Proposition 3.1]{FourBall} and \cite[Proposition 24]{tbbounds}. 
\end{proof}

\begin{remark} Changing $[\SSurf]$ to the class $[S']$ of a different rational Seifert surface changes $\tau_{\alpha}(Y,K)$ according to Remark \ref{rmk:dependence}.  However, changing $[\SSurf]$ to $[S']$ also changes the labeling of $\spinct_m^\spinc\in \SpinC(W_{\!-n}(K))$, and the two changes coincide. Indeed, according to Equation \eqref{enumeration}, if we let  $m_S$ and $m_{S'}$ denote the numbers associated to $S$ and $S'$ by a fixed extension of $\spinc\in \SpinC(Y)$ over the $2$-handle cobordism then 
\[ m_S- m_{S'} = \frac{1}{2}\langle c_1(\spinct),[D_S]-[D_{S'}]\rangle= \frac{1}{2}\langle c_1(\spinct),i_*(\frac{1}{q}[S- S'])\rangle = \frac{1}{2q}\langle c_1(\spinc),[S-S']\rangle.\]
\end{remark}

\section{Proof of the relative adjunction inequality}\label{sec:mainproof}

Armed with the tools from the previous section, we can now precisely state and prove the relative adjunction inequality, Theorem \ref{thm:relativeadjunction}. Consider a surface $\Sigma$ properly embedded in a 4-dimensional cobordism $W$ from $Y_1$ to $Y_2$, so that $\partial\Sigma= -K_1\sqcup K_2$ is a pair of rationally null-homologous knots. In the long exact sequence of the pair $(W,\partial W)$, we have $\partial_*[\Sigma]=0\in H_1(\partial W;\Q)$. By exactness, we can therefore lift $[\Sigma]$  to $H_2(W;\Q)$.  The lift has an ambiguity stemming from classes in $H_2(\partial W)$ (again, by exactness), but we can fix a lift by choosing rational Seifert surfaces $S_1$ and $S_2$ for $K_1$ and $K_2$, respectively.  Given such surfaces, we obtain a geometric lift as the homology class of the rational 2-chain  \[ \SigmaSS:= \frac{1}{q_1}S_1+\Sigma-\frac{1}{q_2}S_2,\] where $q_i$ denotes the order of $K_i$ in $H_1(Y_i;\Z)$. To interpret the latter as a rational $2$-cycle, we may need to add an auxiliary $2$-chain realizing a homology between the rational $1$-cycles $\partial \Sigma$ and $\partial (\frac{1}{q_1}S_1-\frac{1}{q_2}S_2)$.  This choice is canonical up to homology, however, as it is provided by a rational 2-chain in $C_2(K_1\sqcup K_2)$.

\begin{theorem} 

Let $W$ be a smooth compact oriented $4$-manifold with $\partial W=-Y_1\sqcup Y_2$ and $K_1\subset Y_1$ and $K_2\subset Y_2$ be rationally null-homologous knots. If $F_{W,\spinct}(\alpha)=\beta\neq 0$, then 
\begin{equation}\label{eqn:reladjunction2}
\langle c_1(\spinct),[\SigmaSS]\rangle+[\SigmaSS]^2+2(\tau_{\beta}(Y_2,[S_2],K_2)-\tau_\alpha(Y_1,[S_1],K_1))\leq 2g(\Sigma)
\end{equation}
where $\Sigma$ is any smooth oriented properly embedded surface with boundary $-K_1\sqcup K_2$,  $S_i$ are rational Seifert surfaces for $K_i$, and $[\SigmaSS]$ is the lift of $[\Sigma]$ to $H_2(W;\Q)$ obtained from $S_i$ as above.

Furthermore, the left side of the inequality is independent of $S_1$ and $S_2$.
\end{theorem}

\noindent In the above statement, we emphasize that all the terms on the left-hand side are, in general, rational numbers.  In the special case that both $K_i$ are null-homologous, however, all the terms will be integral. 

\begin{proof} If $\Sigma$ is disconnected, tube together the components to form a new connected surface with the same genus, which we continue to denote by $\Sigma$. 

Let $\Whandles$ be the 4-manifold obtained from $W$ by attaching 2-handles along $K_1$ and $K_2$ with appropriate framings so that $\Whandles$ is diffeomorphic to $$\Wone\cup W\cup \Wtwo$$ for some positive integers $n_1$ and $n_2$, where framings are equated with integers using the canonical longitude from Definition \ref{defn:longitude}. Then $\Whandles$ is a cobordism from $\Surgone$ to $\Surgtwo$. 

Assume that $n_1$ and $n_2$ are both large enough that Proposition \ref{prop:FourDInterp} holds. Let 
$\spinc_1=\spinct|_{Y_1}$ and $\spinc_2=\spinct|_{Y_2}$. Then Proposition \ref{prop:FourDInterp} implies that for $m_1>\taua-k_{\spinc_1}$ we have $\alpha\in\Image(\Mapone)$ and for $m_2<\taub-k_{\spinc_2}$, we have  $\Maptwo(\beta)\neq 0$. Therefore the composition $$F_{\Whandles}=\Maptwo\circ F_{W,\spinct}\circ \Mapone$$ is non-trivial.  

Let $\SigmaDD$ denote the smoothly embedded closed surface obtained by capping off $\Sigma$ with the cores of the added 2-handles, so that $[\SigmaDD]=[-D_1\cup \Sigma\cup D_2]$. Applying the vanishing theorem, Theorem \ref{thm:cobordismbound}, to $\SigmaDD$ implies 
\begin{equation}\label{eqn:initialbound}
 \langle c_1(\spincr_{m_1}^{\spinc_1}\#\spinct\#\spinct_{m_2}^{\spinc_2}),[\SigmaDD]\rangle+[\SigmaDD]^2\leq 2 g(\SigmaDD)= 2 g(\Sigma).
\end{equation}
Observe that 
\begin{align*}
[\SigmaDD]&=[- D_1 -\frac{1}{q_1}\SSurf_1]+ [\SigmaSS]+ [\frac{1}{q_2}\SSurf_2+ D_2]\\
&=-[\SSurfDone]+[\SigmaSS]+[\SSurfDtwo]
\end{align*}
and 
\begin{align*}
 [\SigmaDD]^2= [\SSurfDone]^2+[\SigmaSS]^2+[\SSurfDtwo]^2.
\end{align*}
 Therefore, we can rewrite the left hand side of the inequality in Equation \eqref{eqn:initialbound} as
\begin{align*}
-\langle c_1(\spincr_{m_1}^{\spinc_1}), [\SSurfDone]\rangle +[\SSurfDone]^2
 +\langle c_1(\spinct), [\SigmaSS]\rangle +[\SigmaSS]^2 +\langle c_1(\spinct_{m_2}^{\spinc_2}), [\SSurfDtwo]\rangle +[\SSurfDtwo]^2\\
=-2 (k_{\spinc_1}+m_1)+\langle c_1(\spinct), [\SigmaSS]\rangle +[\SigmaSS]^2+2(k_{\spinc_2}+m_2).
\end{align*}
Thus, $$2((k_{\spinc_2}+m_2)-(k_{\spinc_1}+m_1))+\langle c_1(\spinct), [\SigmaSS]\rangle +[\SigmaSS]^2\leq 2 g(\Sigma)$$
whenever $-(k_{\spinc_1}+m_1)<-\taua$ and $k_{\spinc_2}+m_2<\taub$. Maximizing the left hand side  gives
\begin{equation}\label{eqn:bound2}
2(\taub-\taua-2)+\langle c_1(\spinct), [\SigmaSS]\rangle +[\SigmaSS]^2\leq 2g(\Sigma).
\end{equation}

To finish the proof, we exploit the additivity of the non-constant terms in our inequality. Let $\Gamma$ be an arc on $\Sigma$ with endpoints in $K_1$ and $K_2$ respectively. Take $d$ copies of $W$ labeled $W_1,\ldots, W_d$. In $W_1$ fix $d-1$ parallel copies of $\Gamma$ labeled $\Gamma_2,\ldots \Gamma_d$. Form the arc sum of $W_1$ and $W_2$ along $\Gamma_2$ in $W_1$ and $\Gamma$ in $W_2$. To the resulting manifold, form an arc sum along $\Gamma_3$ with $\Gamma$ in $W_3$. Continue this process to obtain a connected manifold $\Warcsum$, which is a successive arc sum of the $d$ copies of $W$. This manifold contains a surface $\Sigmaarcsum$ that is the result of arc summing $d$ copies of $\Sigma$ with itself along the $\Gamma$ arcs. This surface has boundary $-\#^d K_1\sqcup \#^d K_2$ and genus given by $g(\Sigmaarcsum)=dg(\Sigma)$.

 Now note that $\#^d K_i$ has a rational Seifert surface obtained by banding $d$ copies of the rational Seifert surface $S_i$ for $K_i$;   see Remark \ref{rmk:Seifertsum}. Let $[\SigmaSSd]$ denote the class obtained by capping off the ends of $\Sigmaarcsum$ with these rational Seifert surfaces  for $\#^dK_i$. Naturality of Chern classes, together with a Mayer-Vietoris argument, shows  that $$\langle c_1(\#^d\spinct), [\SigmaSSd]\rangle =d\langle c_1(\spinct), [\SigmaSS]\rangle$$ and $$[\SigmaSSd]^2=d[\SigmaSS]^2.$$

The product formula for arc sums,  Theorem \ref{thm:Kun}, applied to $\Warcsum$ shows that the map on Floer homology in the $\SpinC$ structure $\#^d\spinct$ satisfies $F_{\Warcsum}(\alpha^{\otimes d})=\beta^{\otimes d}$.    This allows us to apply Equation \eqref{eqn:bound2} to $\Warcsum$, yielding 
\begin{align*} 2(\taubd-\tauad-2)+\langle c_1(\#^d\spinct), [\SigmaSSd]\rangle +[\SigmaSSd]^2 \\
=2d(\taub-\taua)-4+d\langle c_1(\spinct), [\SigmaSS]\rangle +d[\SigmaSS]^2\leq 2dg(\Sigma).
\end{align*}
Thus for any choice of $d$ we have, 
$$2(\taub-\taua)-\frac{4}{d}+\langle c_1(\spinct), [\SigmaSS]\rangle +[\SigmaSS]^2\leq 2g(\Sigma).$$
Since all the terms in our inequality are rational, taking $d$ sufficiently large yields Inequality \eqref{eqn:reladjunction2}.

Finally, we demonstrate the independence of the bound in  \eqref{eqn:reladjunction2} on the choices of rational Seifert surfaces. If $S_1'$ and $S_2'$ are different choices of rational Seifert surfaces for $K_1$ and $K_2$, respectively, then $$[\Sigma_{S'_1,S'_2}]=\frac{1}{q_1}[S_1'-S_1]+[\SigmaSS]+\frac{1}{q_2}[S_2-S'_2]$$ where $S_1'-S_1\subset Y_1\times I$ and $S_2-S_2'\subset Y_2\times I$.
Since $[S_i-S_i']^2=0$ in $Y_i\times I$,
\begin{align*}
\langle c_1(\spinct),[\Sigma_{S'_1,S'_2}]\rangle+[\Sigma_{S'_1,S'_2}]^2=&\langle c_1(\spinct),[\SigmaSS]\rangle+[\SigmaSS]^2\\&+\frac{1}{q_1}\langle c_1(\spinc_1),[S_1'-S_1]\rangle  -\frac{1}{q_2} \langle c_1(\spinc_2),[S_2'-S_2]\rangle.
\end{align*} 
On the other hand, by Remark \ref{rmk:dependence} $$2\tau_\beta(Y_2,[S'_2],K_2)=2\tau_\beta(Y_2,[S_2],K_2)+\frac{1}{q_2}\langle c_1(\spinc_2),[S_2'-S_2]\rangle $$ and, $$2\tau_\alpha(Y_1,[S_1'],K_1)=2\tau_\alpha(Y_1,[S_1],K_1)+\frac{1}{q_1}\langle c_1(\spinc_1),[S_1'-S_1]\rangle. $$\end{proof}

We also have a corresponding dual statement, which could be useful in applications involving the contact invariant (see \cite{tbbounds,4Dtight}).
\begin{theorem}
Let $W$ be a smooth compact oriented 4-manifold with $\partial W=-Y_1\sqcup Y_2$. Let $K_1\subset Y_1$ and $K_2\subset Y_2$ be rationally null-homologous knots. If $F^*_{W,\spinct}(\varphi)=\psi$ then $$\langle c_1(\spinct),[\SigmaSS]\rangle+[\SigmaSS]^2+2(\tau^*_\varphi(Y_2,[S_2], K_2)-\tau^*_\psi(Y_1, [S_1], K_1))\leq 2g(\Sigma)$$
where $\Sigma$ is any smooth oriented properly embedded surface with boundary $-K_1\sqcup K_2$.
\end{theorem}
\noindent As above, the sum on the left side is independent of the choices of $S_1$ and $S_2$.

\begin{proof} Let $W^\dagger$ denote the 4-manifold $W$, viewed as a cobordism from $-Y_2\to -Y_1$ instead of $Y_1\to Y_2$. Then by Theorem \ref{thm:relativeadjunction}, if $F_{W^\dagger,\spinct}(\varphi)=\psi$, $$\langle c_1(\spinct),[\SigmaSS]\rangle+[\SigmaSS]^2+2(\tau_\psi(-Y_1, [S_1], K_1)-\tau_\varphi(-Y_2, [S_2], K_2))\leq 2g(\Sigma).$$ 
The result now follows from \cite[Theorem 3.5]{HolDiskFour}, which indicates  $F_{W^\dagger,\spinct}=F^*_{W,\spinct}$, together with   Proposition \ref{prop:dualtau}.
\end{proof}

\section{Applications and Examples}\label{sec:applications} In this section we explore specific instances of the relative adjunction inequality, and their consequences. We begin by considering the trivial cobordism $Y\times I$, and  pointing out some immediate corollaries of Theorem \ref{thm:relativeadjunction}: concordance invariance, ``slice-genus" bounds, and crossing change inequalities.  In Subsection \ref{subsec:SoneStwo} we turn to the next simplest cobordisms: boundary connected sums of copies of $D^2\times S^2$ and $S^1\times B^3$, and knots in their boundary $\#^\ell\SoneStwo$.  There, we also establish a general inequality  for $\tau$ invariants under the $H_1(Y)/\mathrm{Tor}$ action on Floer homology (Proposition \ref{prop:H1}), and use this to give bounds on the minimal  geometric intersection number of knots in a given concordance class with the essential $2$-sphere in $\SoneStwo$ (Proposition \ref{prop:monotone2}). In Subsection \ref{subsec:links} we use our understanding of the inequality for connected sums of $\SoneStwo$ in conjunction with the ``knotification" procedure to produce invariants of links, and establish their properties listed in Theorem \ref{thm:links}.  In addition, we use the concordance intersection number bound mentioned above to show that there are knots in $\SoneStwo$ which are not concordant to knotified links, Proposition \ref{prop:notknotified}.  We conclude with Subsection \ref{subsec:comparison}, where we compare our invariants of links with those introduced by Cavallo and \os-Stipsicz in the special case of grid diagrams. 
\subsection{The case of $Y\times I$}\label{subsec:YtimesI}
The trivial cobordism $Y\times I$ induces the identity map on  Floer homology.  Consequently, the relative adjunction inequality yields genus information in $Y\times I$ for {\em any} non-trivial Floer class.  The following corollaries of Theorem \ref{thm:relativeadjunction} are immediate.

\begin{cor}[Concordance invariance]\label{cor:ConcordanceInvariants}
Let $\alpha$ be a nontrivial Floer class in $\HFa(Y)$. If $K_1$ and $K_2$ are concordant in $Y\times [0,1]$ then $\tau_\alpha(Y,K_1)=\tau_{\alpha}(Y,K_2)$.
\end{cor}

\begin{cor}[Slice-genus bounds]\label{cor:SlicegenusBounds} 
Let $\alpha$ be a nontrivial Floer class in $\HFa(Y)$. Then \[|\tau_\alpha(Y,K)|\le g(\Sigma),\]
where $\Sigma\subset Y\times [0,1]$ is any smoothly embedded ``slice" surface with $\partial \Sigma= K\subset Y\times\{1\}$.
\end{cor}

\noindent The inequality above shows that the genus bounds we obtain for surfaces with boundary $K$ in $Y\times I$ are better than any other $4$-manifold with $Y\subset \partial W$.\footnote{Here, we are implicitly assuming the relative homology class of the surface in $W$ equals that of the Seifert surface used to define $\tau$, under the inclusion induced map $H_2(Y,K)\rightarrow H_2(W,K)$.}  This should come as no surprise, since any smooth and proper embedding of a surface in  $Y\times I$ induces an embedding in $W$ by the collar neighborhood theorem.

We can also apply Theorem \ref{thm:relativeadjunction} to yield a general crossing change inequality for our invariants.
\begin{figure}
\includegraphics[height=1.2 in]{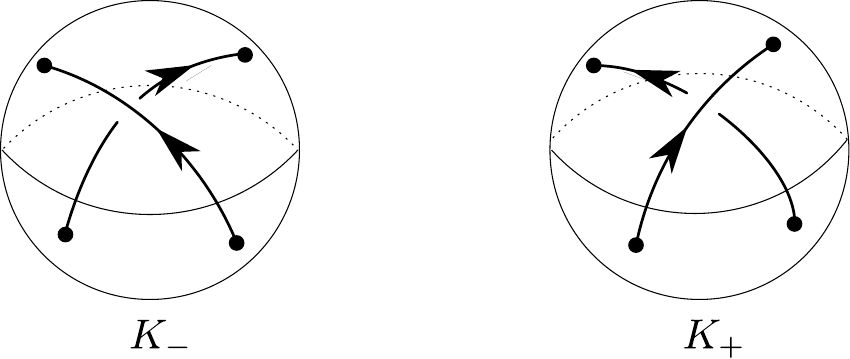}
\caption{Knots differing by a crossing change}
\label{fig:crossingchange}
\end{figure}
\begin{prop}[Crossing change inequalities] \label{prop:crossingchange} Let $K_+, K_-\subset Y$ be rationally null-homologous knots  that are equal outside a 3-ball, in which they differ by a crossing change as in Figure \ref{fig:crossingchange}.  Then for any $\alpha\in \HFa(Y)$ 
\[ \tau_\alpha(Y,[S_-],K_-)\le \tau_\alpha(Y,[S_+],K_+)\le \tau_\alpha(Y,[S_-],K_-)+1,\]
\noindent where the relative homology classes used to define the Alexander grading are represented by rational Seifert surfaces which agree outside the ball.
\end{prop}
\begin{proof} There is a smooth genus one cobordism in $Y\times [0,1]$ between $K_+$ and $K_-$ obtained by attaching a band to the incoming knot to change the crossing, followed by an additional band that rejoins the additional meridional component.  Applying Theorem \ref{thm:relativeadjunction} to this cobordism shows $|\tau_\alpha(Y,K_+)-\tau_\alpha(Y,K_-)|\le 1$.  

To establish $\tau_\alpha(Y,K_-)\le \tau_\alpha(Y,K_+)$, observe that $K_+$ is concordant to $K_-$ inside $Y\times [0,1]$, blown up in the interior. The blow-up formula, \cite[Theorem 3.7]{HolDiskFour}, implies there are two  $\SpinC$ structures on $\widehat{Y}=(Y\times [0,1])\#\overline{\CP{2}}$ that induce the identity map on $\widehat{HF}(Y)$.  Hence, $F_{\widehat{Y}}(\alpha)=\alpha$ for these $\SpinC$ structures, and we can apply Theorem \ref{thm:relativeadjunction}.  Since the concordance intersects the exceptional sphere zero times algebraically, the  terms on the left side of  the relative adjunction inequality not involving $\tau$ vanish.\end{proof}

\subsection{Genus bounds for knots in $\#^\ell\SoneStwo$}\label{subsec:SoneStwo}
There are some particularly simple $4$-manifolds bounded by $\#^\ell\SoneStwo$ whose maps on Floer homology are understood. In this section we apply the relative adjunction inequality to study the genus problem in this context.  

We first describe a family of $4$-manifolds, labeled by elements in Floer homology.  To do so,  fix an ordering on the components of the connect sum $\#^\ell\SoneStwo$. Under the K{\"u}nneth Theorem, generators of $\HFa(\#^\ell\SoneStwo)$ are of the form $\theta_{\epsilon_1}\otimes \ldots\otimes \theta_{\epsilon_\ell}$ for $\epsilon_i\in\{+,-\}$. For each element $\Theta=\theta_{\epsilon_1}\otimes \ldots\otimes \theta_{\epsilon_\ell}$ there is an associated 4-manifold $W_{\Theta}= W_{\epsilon_1}\natural \ldots \natural W_{\epsilon_\ell}$ where $W_+=S^1\times B^3$ and $W_-=D^2\times S^2$.  
\begin{cor} \label{cor:S1S2genusbound}
Let $W_\Theta$ be the boundary connected sum of copies of $S^1\times B^3$ and $D^2\times S^2$, specified by $\Theta$, as above, and
 $K$ a knot in the boundary of $W_\Theta$. Then $$\tau_\Theta(K)\leq g(\Sigma),$$
 where $\Sigma\subset W_\Theta$ is any smooth and properly embedded surface with boundary $K$.
\end{cor}
\begin{proof}
Suppose $\Sigma\subset W_\Theta$ is a properly embedded oriented surface with boundary $K$ in $\#^\ell\SoneStwo$. Remove a 4-ball from $W_\Theta$ and tube $\Sigma$ to the $S^3$ boundary component. This gives a cobordism from the unknot $U$ to $K$ inside $W_\Theta-B^4$. 

The notation for $W_\Theta$ reflects that the map on Floer homology induced by $W_\Theta-B^4$ sends the generator of $\HFa(S^3)$ to $\Theta\in \HFa(\#^\ell\SoneStwo)$.  Indeed, this follows from the definition of the maps associated to $4$-dimensional $1$-handle attachment \cite[Section 4.3]{HolDiskFour} and a calculation of the map induced on Floer homology by attaching a $2$-handle along a zero framed unknot.  This latter calculation can be done directly via adapted Heegaard triple diagrams, or using the surgery exact triangle together with the grading shift formula.   Thus, $\Theta\in \Image(F_{W_{\Theta}-B^4,\spinct_0})$.

Applying Theorem \ref{thm:relativeadjunction} yields: $$\langle c_1(\spinct_0),[\Sigma]\rangle+[\Sigma]^2+2(\tau_{\Theta}(K)-\tau(U))\leq 2g(\Sigma).$$ 
Now $\tau(U)=0$, $\langle c_1(\spinct_0),[\Sigma]\rangle=0$ and $[\Sigma]^2=0$. The result follows.
\end{proof}
\begin{example}\label{example:Whiteheadslice}
Our calculation in Section \ref{Whitehead},  together with Corollary \ref{cor:S1S2genusbound}, implies that the positive Whitehead knot is not slice in $S^1\times B^3$, since $\tautop( Wh^+)=1$. On the other hand, $Wh^+$ bounds a disk $D^2\times S^2$.  Indeed, viewed as a knot in $S^3$, $Wh^+$ bounds a twisted disk which intersects the zero-framed unknot in two points, from which one can obtain a smoothly embedded disk in $D^2\times S^2$ bounded by $Wh^+$. 

That $Wh^+$ doesn't bound a disk in $S^1\times B^3$ can also be seen with far less sophisticated techniques, and indeed one can show it does not even bound a locally flat disk.  Perhaps the first treatment of this can be attributed to Goldsmith \cite[Page 136]{Goldsmith}, who obstructs null-concordance of the Whitehead link using linking numbers between the lifts of one component to the infinite cyclic cover of the other, an approach which also obstructs null-concordance of $Wh^+$ in $\SoneStwo\times [0,1]$ (which, in turn obstructs sliceness in $S^1\times B^3$).   Closely related is Wall's self-intersection number over $\Z[\Z]$, which can be used to provide an invariant of immersed disks bounded by $Wh^+$ in $S^1\times B^3$ which obstructs finding a locally flat embedded disk \cite{Wall}.   Equivalently, one can use Schneiderman's concordance invariants for knots in $\SoneStwo$, which stem from Wall's intersection number \cite{Schneiderman}.   
\end{example}

For knots in  $\#^\ell\SoneStwo$ the invariants $\tau_{\Theta}$ depend on the ordering of the $\epsilon_i$. While the ordering of the $\epsilon_i$ does not impact the diffeomorphism type of $W_{\Theta}$, the genus bound corresponding to $\tau_{\Theta}$ does depend on the ordering of the factors of $W_{\Theta}$.   Indeed, while a diffeomorphism between  $W_{\Theta}$ and $W_{\Theta'}$ induces a diffeomorphism between their boundaries, this map may not preserve a given knot.  Our invariants can therefore potentially distinguish the surfaces a knot bounds in, say, $(S^1\times B^3)\natural (D^2\times S^2)$ from those it bounds in $(D^2\times S^2)\natural (S^1\times B^3)$.  The following example illustrates this point.

\begin{figure}
\includegraphics[height=1.2in]{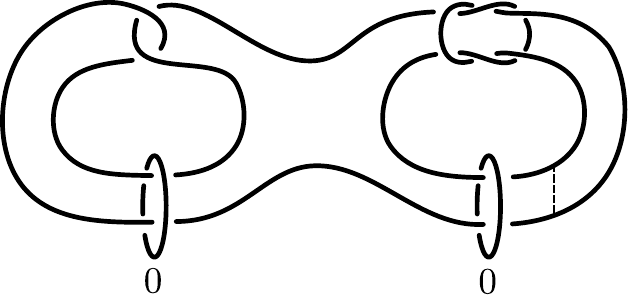}
\caption{$(\SoneStwo,Wh^+)\#(\SoneStwo, K)$.}
\label{connectsum}
\end{figure}

\begin{example}
Consider the knot $(\SoneStwo, Wh^+)\#(\SoneStwo, K)$ shown in Figure \ref{connectsum}. The  band attachment  indicated by the dotted arc in the figure provides a pair of pants cobordism from $K$ to a  two component unlink. Attaching a disk to one component of the unlink shows $K\subset  \SoneStwo$ is concordant in $\SoneStwo\times I$ to the unknot, hence $\tau_{\theta_+}(K)=\tau_{\theta_-}(K)=0$.  Additivity of $\tau$ under connected sum shows
\begin{align*}
\tau_{\theta_-\otimes \theta_+}(Wh^+\#K)&=0 &\text{and}& &\tau_{\theta_+\otimes \theta_-}(Wh^+\#K)&=1.
\end{align*}
Thus, $Wh^+\#K$ is not slice in $(S^1\times B^3)\natural (D^2\times S^2)$. On the other hand, the construction of the slice disk in Example \ref{example:Whiteheadslice}, together with the above, shows that $Wh^+\#K$ is slice in $(D^2\times S^2)\natural(S^1\times B^3)$.
\end{example}

The $\tau_\Theta$ invariants for knots in $\#^\ell\SoneStwo$ satisfy a type of monotonicity.  
\begin{prop}\label{prop:monotone} Let $\Theta=\theta_{\epsilon_1}\otimes \ldots\otimes \theta_{\epsilon_\ell}$ and $\Theta'=\theta_{\epsilon'_1}\otimes \ldots\otimes \theta_{\epsilon'_\ell}$, and suppose $\epsilon_i\le \epsilon'_i$ for all $i$, where we order signs by $-< +$.  Then   $\tau_\Theta(\#^\ell\SoneStwo, K)\le \tau_{\Theta'}(\#^\ell\SoneStwo, K)$ for all knots $K$.
\end{prop}

The proposition will follow from a more general inequality concerning  the $H_1(Y)/\mathrm{Tor}$-action on Heegaard Floer homology:
\begin{prop}\label{prop:H1} Let $Y$ be a  3-manifold, and $\gamma\in H_1(Y)/\mathrm{Tor}$ be a class in the free part of its first homology.  If $A_\gamma(\alpha)=\beta$, then we have
\[ \tau_\beta(Y,K)\le \tau_\alpha(Y,K)\]
for any knot $K$.  Here, $A_\gamma$ denotes the action of $\gamma$ on the Floer homology of $Y$.
\end{prop}
\begin{proof}
Pick an immersed curve on a pointed Heegaard diagram  for the (pointed) 3-manifold whose associated homology class represents $\gamma\in H_1(Y)/\mathrm{Tor}$.  By abuse of notation, denote this curve by $\gamma$ as well. The  $H_1(Y)/\mathrm{Tor}$ action is defined by chain maps $a_\gamma$, specified on generators by:

$$a_\gamma (\x)=\sum_{\y\in\Ta\cap\Tb}\sum_{\substack{ \phi\in\pi_2(\x,\y)\\   \mu(\phi)=1}} \big(\gamma\cdot(\partial_{\alpha}\phi)\big) \#\widehat{\mathcal{M}}(\phi)\cdot\y$$
where $\gamma\cdot(\partial_{\alpha}\phi)$ is the algebraic intersection number of $\gamma$ with the subset of the boundary of the domain of the Whitney disk lying on the $\alpha$ curves.  Here, we consider the action on $\CFa(Y)$, so we count only pseudo-holomorphic Whitney disks that avoid the hypersurface specified by the basepoint, i.e. $n_w(\phi)=0$.    Since $n_z(\phi)\ge 0$ for any disk counted in the operator above, it follows that \[A(\x)-A(\y)=n_z(\phi)-n_w(\phi)\ge 0\]
where $(\x,\y)$ is any pair of generators such that $\y$ appears with non-zero coefficient in $a_\gamma(\x)$.  Passing to homology, the result follows.   
\end{proof}

\noindent Proposition \ref{prop:H1} implies the  inequalities for $\tau_\Theta$ given in Proposition \ref{prop:monotone}, since the Floer homology of  $\#^\ell\SoneStwo$  is isomorphic, as a module over $H_1/\mathrm{Tor}$, to $\Lambda^* H^1(\#^\ell\SoneStwo)$, with module structure given by the pairing between homology and cohomology.  In this setting, or the general situation where a 3-manifold contains an $S^1\times S^2$ connected summand in its prime decomposition, we also obtain bounds in the opposite direction in terms of the geometric intersection number of the knot with the 2-sphere.  
\begin{prop} \label{prop:monotone2} Let $\alpha\in \HFa(Y)$, and consider $\alpha\otimes \theta_{\pm} \in \HFa(Y\#S^1\times S^2)$, under the isomorphism $\HFa(Y\#S^1\times S^2)\cong \HFa(Y)\otimes \HFa(S^1\times S^2)$.  Then for any rationally null-homologous $K\subset Y\# S^1\times S^2$
\[  \tau_{\alpha\otimes \theta_+}(Y\# S^1\times S^2, K) \le \tau_{\alpha\otimes \theta_-}(Y\# S^1\times S^2, K) + N,\]
where $N$ is half the geometric intersection number between $K$ and the sphere in $S^1\times S^2$.
\end{prop}
\begin{proof} Assume the sphere and $K$ have been isotoped to be transverse and to have minimal intersection number.  Since $K$ is rationally null-homologous, the algebraic intersection number with the sphere is zero, hence the intersections come in pairs of opposite signs.  Let $N$ be the number of such pairs.  Attach $N$ bands to $K$ along arcs pairing these points, to arrive at a link that is disjoint from a neighborhood of the the 2-sphere.  Now attach a 4-dimensional 3-handle along this neighborhood, yielding a cobordism with outgoing end diffeomorphic to $Y$.  According to \cite[Section 4.3]{HolDiskFour}, the associated cobordism map sends $\alpha\otimes \theta_-$ to $\alpha\in\HFa(Y)$.  Next attach a 4-dimensional 1-handle to produce a cobordism whose outgoing end is again $Y\# S^1\times S^2$. The cobordism-induced map on Floer homology  \cite[Section 4.3]{HolDiskFour} maps $\alpha$  to $\alpha\otimes \theta_+$.  Finally, attach $N$ bands to the link to recover the original knot $K\subset Y\# S^1\times S^2$.  This produces a knot cobordism from $K$ to itself of genus $N$, in a 4-dimensional cobordism $X$, whose map on Floer homology satisfies $F_X(\alpha\otimes \theta_-)=\alpha\otimes \theta_+$.  Theorem \ref{thm:relativeadjunction} implies:
\[ 2(\tau_{\alpha\otimes \theta_+}(Y\# S^1\times S^2, K)-\tau_{\alpha\otimes \theta_-}(Y\# S^1\times S^2, K))\le 2N,\]
yielding the desired inequality.
\end{proof}

\begin{remark}  It is interesting to compare our invariants with recent work of Manolescu-Marengon-Sarkar-Willis \cite{manolescu2019generalization},  which develops invariants for links in $\#^\ell\SoneStwo$ from versions of Khovanov and Lee homology adapted to this setting.  Their invariants satisfy similar genus bounds in $W_\Theta$ and  agree for many knots and links. They will differ in general, however, due to the behavior of both theories under connect sums and the fact that they differ in the special case of knots in $S^3$ \cite{STau}. \end{remark}

\subsection{Knotification and Invariants of Links}\label{subsec:links} Our genus bounds for null-homologous knots in $\#^\ell\SoneStwo$ produce genus bounds for links in $S^3$ via Ozsv\'ath and Szab\'o's knotification construction. We summarize their construction below. See \cite[Section 2.1]{Knots} for more details.

Let $L\subset Y$ be an oriented $|L|$-component link in a 3-manifold $Y$. The \emph{knotification} of $L$ is an oriented knot $\kappa(L)$ in $Y\#^{|L|-1}\SoneStwo$ formed in the following way. The idea is to turn $L$ into a knot by attaching bands that connect all the components. To make this well-defined,  before banding a pair of link components together, first fix a pair of points, one on each component, and attach a 4-dimensional 1-handle to $Y\times I$ along $Y\times \{1\}$, whose attaching region $S^0\times B^3$ is identified with a neighborhood of the pair of points. Now, band the components together via a band that passes through the 1-handle. Since $L$ has $|L|$ components, choosing our bands optimally produces a knot $\kappa(L)$ after $|L|-1$ band attachments.  By isotopy of the attaching regions and  handleslide amongst the 1-handles and bands,  one sees that $\kappa(L)\subset Y\#^{|L|-1}\SoneStwo$ is well-defined up to diffeomorphism  \cite[Proposition 2.1]{Knots}.

Now, given a link $L$ in a 3-manifold $Y$ and a pair of nonzero elements $\alpha$ in $\HFa(Y)$ and $\Theta$ in $\HFa(\#^{|L|-1}\SoneStwo)$ define $$\tau_{\alpha\otimes\Theta}(Y,L):=\tau_{\alpha\otimes\Theta}(Y\#^{|L|-1}\SoneStwo,\kappa(L)).$$ For links in the 3-sphere, we denote these invariants simply by $\tau_\Theta(L)$, since $\HFa(S^3)$ contains a single nontrivial element. 

The $\tau$ invariants of links inherit an additivity property from the $\tau$ invariants of their knotifications. 

\begin{prop}[Additivity]
Given a pair of links $L_1\subset Y_1$ and $L_2\subset Y_2$, $$\tau_{\alpha_1\otimes \Theta_1}(Y_1,L_1)+\tau_{\alpha_2\otimes \Theta_2}(Y_2,L_2)=\tau_{\alpha_1\otimes\Theta_1\otimes\alpha_2\otimes \Theta_2}(Y_1\#Y_2,L_1\#L_2).$$
Here $L_1\#L_2$ denotes the link resulting from the connect sum of any component of $L_1$ with any component of $L_2$.
\end{prop}
\begin{proof}
First, we claim that $\kappa(L_1\#L_2)$ is isotopic to $\kappa(L_1)\#\kappa(L_2)$. To see this, observe that $L_1\#L_2$ is constructed by first forming the connect sum $Y_1\#Y_2$ and then band summing $L_1$ and $L_2$ together along an arc passing through the separating 2-sphere. The resulting link has $|L_1|+|L_2|-1$ components. To form its knotification, $\kappa(L_1\#L_2)$, we must choose a collection of $|L_1|+|L_2|-2$ arcs along which to knotify. Since the result does not depend on the choice of arcs, we can choose a collection of arcs, none of which intersect the separating 2-sphere in $Y_1\#Y_2$. 

On the other hand, using the same collection of arcs as above, we can first form $\kappa(L_1)$ and $\kappa(L_2)$ and then form the connect sum. The resulting knots are isotopic since we have used the same collection of arcs in both constructions. 

The result now follows from the additivity of the $\tau$ invariants of the knotifications.
\end{proof}

%

We can easily extend the results from Subsection \ref{subsec:YtimesI} to the case of links.  Corollary \ref{cor:ConcordanceInvariants} implies that if a pair of knotified links are concordant, then their $\tau$ invariants coincide.  If a pair of links in $Y$ are concordant, Hedden and Kuzbary \cite{Hedden-Kuzbary} describe how to surger the concordance to yield a concordance between their knotifications.  Theorem \ref{thm:links}\ref{thm:links1} follows immediately.

\begin{cor}[Concordance invariance \cite{Hedden-Kuzbary}]\label{cor:ConcordanceInvariance}
If links $L\subset Y$ and $L'\subset Y$ are concordant, then for any choices of $\alpha\in \HFa(Y)$ and $\Theta\in \HFa(\#^{|L|-1}\SoneStwo)$  $$\tau_{\alpha\otimes\Theta}(Y,L)=\tau_{\alpha\otimes\Theta}(Y,L').$$  
\end{cor}

Observe that if two links differ by a crossing change, then so do their knotifications.  This gives the crossing change inequalities stated in Theorem \ref{thm:links}\ref{thm:links-crossing}.

\begin{cor}[Crossing Change inequalities] \label{cor:Crossingchange} If $L_-,L_+\subset Y$  differ at a single crossing, which is positive  in $L_+$ and negative in $L_-$,\[ \tau_{\alpha\otimes \Theta}(Y,L_-)\le \tau_{\alpha\otimes \Theta}(Y,L_+)\le \tau_{\alpha\otimes\Theta}(Y,L_-)+1,\]
where the relative homology classes of Seifert surfaces used in the knotifications (suppressed) agree, as before.
\end{cor}

We also recover the slice-genus bounds for links stated in Theorem \ref{thm:links}\ref{thm:links2}:
\begin{prop}[Slice-genus bounds]  \label{prop:SlicegenusBounds-links} If  $\Sigma\subset Y\times[0,1]$ is a smoothly embedded surface with boundary $ L\subset Y\times\{1\}$ its Euler characteristic satisfies \[2|\tau_{\alpha\otimes\Theta}(Y,L)|\le |L|-\chi(\Sigma),\]
for any choices of $\alpha\in \HFa(Y)$ and $\Theta\in \HFa(\#^{|L|-1}\SoneStwo)$.  
\end{prop}
\begin{proof} By attaching $1$- and $2$-handles to the outgoing end of $Y\times[0,1]$ (with the latter attached along $0$-framed unknots), we can find cobordisms from $Y$ to $Y\#^{|L|-1}\SoneStwo$ which map $\alpha$ to $\alpha\otimes \Theta$ for any generating decomposable tensor   $\Theta=\theta_{\epsilon_1}\otimes \ldots\otimes \theta_{\epsilon_{|L|-1}}$  (equivalently, we can take an arc sum of $W_\Theta-B^4$ with $Y\times [0,1]$). Attach bands to $\Sigma$ in the $1$- and $2$-handles to yield a cobordism $\Sigma_\Theta$ from $L$ to $\kappa(L)$ of Euler characteristic equal to minus the number of bands, $\chi(\Sigma_\Theta)=1-|L|$.   Puncturing $\Sigma\cup \Sigma_\Theta$ and tubing the new boundary to the incoming end $Y\times\{0\}$, yields a cobordism from the unknot $U$ to $\kappa(L)$ with Euler characteristic equal to $\chi(\Sigma\cup \Sigma_\Theta)-1=\chi(\Sigma)+1-|L|-1$.  Theorem \ref{thm:relativeadjunction} implies 
\[  2(\tau_{\alpha\otimes\Theta}(Y\#^{|L|-1}\SoneStwo,\kappa(L))- \tau_\alpha(Y,U))\le |L|-\chi(\Sigma)\]
The above holds for any generating decomposable tensor.  To obtain the inequality for arbitrary $\Theta$, we observe that such a vector can be written as a sum of generating elements to which the inequality applies, and then appeal to  the subadditivity of $\tau$ invariants, Proposition \ref{prop:sum}. 

 For the reverse inequality, note that we can alternatively produce a cobordism from $Y\#^{|L|-1}\SoneStwo$ to $Y$ which maps  $\alpha\otimes \Theta$ to $\alpha$, again for any choices of $\alpha$ and decomposable tensor $\Theta=\theta_{\epsilon_1}\otimes \ldots\otimes \theta_{\epsilon_{|L|-1}}$.  Such a cobordism is obtained from $(Y\#^{|L|-1}\SoneStwo)\times [0,1]$ by attaching $3$- and $2$-handles to kill the apparent 2-spheres, with the choice of handle used to kill a 2-sphere determined by the decomposable tensor.   This shows $-2\tau_{\alpha\otimes \Theta}(Y,L)\le |L|-\chi(\Sigma)$ for such $\Theta$. For an arbitrary non-zero class $\Theta$, we appeal to the monotonicity of $\tau$ invariants with respect to the $H_1/\mathrm{Tor}$ action, Proposition \ref{prop:H1}.  Given any non-zero $\Theta$ one can find a sequence of curves $\gamma_1,...,\gamma_n$ for which $A_{\gamma_n}\circ\ldots\circ A_{\gamma_1}(\alpha\otimes\Theta)=\alpha\otimes\Theta_{bot}$.  This shows that $\tau_{\alpha\otimes \Theta_{bot}}\le \tau_{\alpha\otimes \Theta}$ which, when negated, yields the reverse inequality for arbitrary $\Theta$ (see the proof of Proposition \ref{prop:Monotonicity} for further details) \end{proof}

\noindent  We define the smooth ``slice-genus" of an oriented link in $S^3$ to be
\[ g_4(L):= \mathrm{min} \left\{\ \frac{|L|-\chi(\Sigma)}{2} \ \Big| \ \Sigma\hookrightarrow B^4, \mathrm{smooth, with}\ \partial \Sigma=L\right\}.\]
\noindent This definition of genus may seem strange, placed in comparison to the standard notion of the number of ``holes".  However, the definition here  better   captures distinctions in complexity (in the spirit of the Thurston norm), and is more tightly connected to Floer-type invariants, e.g. \cite[Theorem 1.1]{NiKnotLink}.  Indeed, specializing to links in $S^3$, Proposition \ref{prop:SlicegenusBounds-links} shows that all the $\tau$ invariants of links produce slice-genus bounds.
\begin{cor}
If $L\subset S^3$ is an oriented link, and $\Theta$ any class in $\HFa(\#^{|L|-1}\SoneStwo)$ $$|\tau_\Theta(L)|\leq g_4(L).$$
\end{cor}

\noindent Different choices of $\Theta$ potentially give different genus bounds. However, the bounds obtained via $\Theta_{top}$ and $\Theta_{bot}$ will always be the best.  We make this precise with the  monotonicity property,  Theorem \ref{thm:links}\ref{thm:links-monotoncity}, which we now establish.

\begin{prop}[Monotonicity] \label{prop:Monotonicity} If $\Theta'=\iota_x(\Theta)$, where  $\iota_x$ denotes the interior product with a class $x\in H_1(\mathbb{T}^{|L|-1})$, then 
\[ \tau_{\alpha\otimes \Theta'}(Y,L)\le \tau_{\alpha\otimes \Theta}(Y,L)\le \tau_{\alpha\otimes\Theta'}(Y,L)+1.\]
In particular, if $\tautop(L)$ and $\taubot(L)$ denote the invariants corresponding to the unique elements in $H^*(\mathbb{T}^{|L|-1})$ of maximal and minimal grading, respectively, then
\[ \taubot(L)\le \tau_\Theta(L)\le \tautop(L)\le \taubot(L)+|L|-1.\]
\end{prop}
\begin{proof}  The $H_1/\mathrm{Tor}$ action on Floer homology is natural with respect to connected sums in the following sense. If ${\gamma_1\cup\gamma_2}$ is represented by a curve which decomposes as a union along a connected sum of Heegaard diagrams,  then under the resulting isomorphism of complexes  $$\widehat{CF}(Y_1\#Y_2,\mathfrak s_1\#\mathfrak s_2)\cong
\widehat{CF}(Y_1,\mathfrak s_1)\otimes\widehat{CF}(Y_2,\mathfrak
s_2)$$ provided by the  K{\"u}nneth  theorem,  the action of ${\gamma_1\cup\gamma_2}$ is given by
\[a_{\gamma_1\cup\gamma_2}=a_{\gamma_1}\otimes
\mathrm{id}+\mathrm{id}\otimes a_{\gamma_2}.\]
  This can be used to show that the identification of $\HFa(\#^{|L|-1}\SoneStwo)$ with $H^*(\mathbb{T}^{|L|-1})$ intertwines the $H_1(\#^{|L|-1}\SoneStwo)/\mathrm{Tor}$ action with the action by interior product, and that $x\in H_1(\mathbb{T}^{|L|-1})$ acts on $\alpha\otimes \Theta$ by $(a_x)_*(\alpha\otimes \Theta)=\mathrm{Id}(\alpha)\otimes \iota_x(\Theta)=\alpha\otimes \Theta'$.  The left-hand inequality in the first part now follows immediately from Proposition \ref{prop:H1}.
  
To show $\tau_{\alpha\otimes \Theta}(Y,L)\le \tau_{\alpha\otimes\Theta'}(Y,L)+1$, we observe that a knotified link intersects each essential 2-sphere created by the 1-handle attachments in at most two points (arising from where the band passes through the handles).  We then appeal to Proposition \ref{prop:monotone2}.   Iterating the inequalities in the first line and using the fact that $\Theta_{top}$ generates $\HFa(\#^{|L|-1}\SoneStwo)$ as a module with respect to the $H_1/\mathrm{Tor}$ action yields the inequalities for links in $S^3$ stated in the second line in the case that $\Theta=\theta_{\epsilon_1}\otimes \ldots\otimes \theta_{\epsilon_{|L|-1}}$.  The  inequality for general $\Theta$ follows from Proposition \ref{prop:sum}, and the fact that any non-zero class maps to $\Theta_{bot}$ under iteration of  the $H_1/\mathrm{Tor}$ action.
  \end{proof}

In  \cite[Theorem 1.1]{FourBall}, \ons\ establish a general bound for the genera of surfaces in negative definite $4$-manifolds bounded by a knot in the $3$-sphere. Armed with the relative adjunction inequality, we can easily extend their result to our invariants for links.  This is the content of Theorem \ref{thm:links}\ref{thm:links-negdef}:
\begin{theorem}[Definite $4$-manifold bound]\label{thm:linkgenus}
Let $W$ be a smooth, oriented 4-manifold with $b_2^+(W)=b_1(W)=0$, and $\partial W=S^3$. If $\Sigma$ is any smoothly embedded surface in $W$ with boundary a link $L\subset S^3$, then  
$$2\tau_\Theta(L)+[\Sigma]^2+|[\Sigma]|\leq  |L|-\chi(\Sigma),$$
where $|[\Sigma]|$ is the $L_1$-norm of the homology class $[\Sigma]\in H_2(W,\partial W)\cong H_2(W)$.
\end{theorem}

\begin{proof} Suppose $\Sigma\subset W$ is a properly embedded, oriented surface with boundary $L\subset \partial W=S^3$. In addition, suppose $\Sigma_\Theta$ is the cobordism in $W_{\Theta}-B^4$ from $L$ to $\kappa(L)$ obtained by attaching $|L|-1$ bands to $L$. Gluing $W$ to $W_{\Theta}-B^4$ along $S^3$, we form a 4-manifold $\widehat{W}=W\cup_{S^3} (W_{\Theta}-B^4)$  containing the surface $\Sigma \cup \Sigma_\Theta$ with boundary $\kappa(L)$.

Now, fix a $\SpinC$ structure $\spinct$ on $W$ satisfying $c_1(\spinct)=-b_2(W)$ and $\langle c_1(\spinct),[\Sigma]\rangle =|[\Sigma]|$ and observe that any generator $\Theta=\theta_{\epsilon_1}\otimes \ldots\otimes \theta_{\epsilon_{|L|-1}}$ is in the image of \[F_{\widehat{W}-B^4, \spinct\#\spinct_0}=F_{W_\Theta-B^4,\spinct_0}\circ F_{W-B^4,\spinct}.\] This follows from Lemma 3.4 of \cite{FourBall}, which implies that for such $\spinct$ the map $F_{W-B^4,\spinct}$ is nontrivial, together with the argument made in the proof of Corollary \ref{cor:S1S2genusbound}, which implies that $\Theta$ is in the image of $F_{W_{\Theta}-B^4, \spinct_0}$.  

Tubing $\Sigma\cup\Sigma_\Theta$ to the $S^3$ boundary component of $\widehat{W}-B^4$ gives a cobordism from the unknot $U$ to $\kappa(L)$ and applying Theorem \ref{thm:relativeadjunction} we obtain: $$\langle c_1(\spinct\#\spinct_0),[\Sigma\cup\Sigma_\Theta]\rangle+[\Sigma\cup\Sigma_\Theta]^2+2(\tau_\Theta(\kappa(L))-\tau(U))\leq 2g(\Sigma\cup \Sigma_\Theta)= 1-\chi(\Sigma\cup\Sigma_\Theta).$$
Since $\tau(U)=0$, $c_1(\spinct_0)=0$ and $[\Sigma_\Theta]^2=0$, the left side simplifies as 
\begin{align*}
|[\Sigma]|+[\Sigma]^2 +2\tau_{\Theta}(\kappa(L)).
\end{align*}
At the same time, $\chi(\Sigma\cup\Sigma_\Theta)=\chi(\Sigma)-|L|+1.$ 
Thus, $1-\chi(\Sigma\cup\Sigma_\Theta)=|L|-\chi(\Sigma).$ The result follows at once for the generating decomposable tensors, and extends to arbitrary $\Theta$ using subadditivity, Proposition \ref{prop:sum}, as before.
\end{proof}

There are several classes of links where general structural theorems hold for our $\tau$ invariants.  For instance, we can establish Theorem \ref{thm:links}\ref{thm:links-alternating}, which says that the invariants for alternating links, like alternating knots, are determined by their signature: 

\begin{theorem}[Alternating links] \label{thm:altLinks}
Suppose $L\subset S^3$ is an alternating link of $|L|$ components, and $\Theta\in \HFa(\#^{|L|-1}S^1\times S^2)$ is a class with grading $k$.   Then $\tau_\Theta(L)= k-\frac{\sigma}{2}$, where $\sigma(L)$ is the signature. In particular, $\tautop(L)= \frac{|L|-\sigma(L)-1}{2}$ and $\taubot(L)=\frac{-|L|-\sigma(L)+1}{2}$
\end{theorem}
\begin{proof}
This is a straightforward consequence of \cite[Theorem 4.1]{AltKnots}.  That theorem indicates the knot Floer homology of alternating links which, by definition, is the knot Floer homology of their knotification, is determined by the Alexander polynomial and signature in much the same manner as the better known result from loc. cit. regarding knots.  In particular, the knot Floer homology of an alternating link is ``thin", with the Alexander grading $s$ group supported entirely in Maslov grading $s+\frac{\sigma}{2}$.  It follows that, given a Floer homology class $\Theta\in \HFa(\#^{|L|-1}S^1\times S^2)$ of Maslov grading $k=s+\frac{\sigma}{2}$, the only place in the $\Z$-filtered  homotopy type of the complex corresponding to $\kappa(L)$ where such a class can arise is in Alexander grading $s=k- \frac{\sigma}{2}$.  For the statement about $\tautop$ and $\taubot$, we observe that the highest and lowest Maslov gradings supporting non-trivial Floer groups for  $\HFa(\#^{|L|-1}S^1\times S^2)$ occur at $\frac{|L|-1}{2}$ and $\frac{1-|L|}{2}$, respectively.
\end{proof}

The 4-genus bound provided by $\tau$ is known to be sharp for quasipositive knots \cite{Olga2004}.  We can show, more generally, that the 4-genus bound provided by $\tau_{top}$ is also sharp for quasipositive links. Recall, then, that a \emph{quasipositive link} is the closure of a braid of the form  \[\beta=\displaystyle\prod_{k=1}^mw_k\sigma_{i_k}w_k^{-1},\] where $w_k$ is a braid word in the $n$-strand braid group $B_n$ and $\sigma_{i_k}$ denotes a standard generator.  Rudolph introduced this notion in \cite{Rudolph1983}, where he showed that such links arise as the intersection of a plane algebraic curve with the boundary of the bidisk $D^2\times D^2\subset \C^2$.  In \cite{Rudolph1983-2}, Rudolph explains how to push these algebraic curves into the boundary $3$-sphere to a positively braided ribbon surface.   A link is called \emph{strongly quasipositive} if it bounds a positively braided ribbon surface which is embedded, i.e. is a Seifert surface. 
\begin{prop}\label{prop:quasipositive}
If $L$ is a quasipositive link in $S^3$ then $\tautop(L)=g_4(L).$ Moreover, if $L$ is strongly quasipositive then $\tautop(L)=g_4(L)=g_3(L).$
\end{prop}

The proof of Proposition \ref{prop:quasipositive} will take a detour through some contact geometric features of the theory, upon which we now embark.  We will primarily relegate our exploration of the interaction between the relative adjunction inequality and contact geometry to another paper (see \cite{4Dtight}), and here draw on only what we need for studying $\tau_{top}(L)$.   

To begin, we recall from \cite{tbbounds} that one can define an invariant of knots in contact 3-manifolds by $\tau_{\xi}(Y,K):=\tau^*_{c(\xi)}(Y,K)$ where $c(\xi)$ denotes the \os\ contact class associated to $\xi$ \cite[Definition 1.2]{Contact}. We will need the following fact, which equates $\tau_{top}(\#^\ell \SoneStwo, K)$ and $\tau_{\xi_{std}}(\#^\ell\SoneStwo, K)$. 

\begin{prop} \label{prop:top/std} Let $K$ be a knot in $\#^\ell\SoneStwo$. Then
$$\tautop(K)=\tau_{\xi_{std}}(K)=\tau^*_{c(\xi_{std})}(K)$$
where $\xi_{std}$ is the unique tight contact structure on $\#^\ell\SoneStwo$.
\end{prop}
\begin{proof} First observe that the classes $\Theta_{top}$ and $c(\xi_{std})$ are both decomposable tensors in the Floer homology of $\#^\ell S^1\times S^2$, under the identification of the latter as an iterated tensor product provided by  the K{\"u}nneth formula \cite[Theorem 1.4]{HolDiskTwo}.  This is immediate from the discussions above for $\Theta_{top}$. For the contact class it follows from the fact that $\xi_{std}$ is the iterated contact connected sum of the unique tight contact structure on $\SoneStwo$, together with the product formula for $c(\xi)$ under contact connected sums, \cite[Property 4, pg. 105]{tbbounds}.   It therefore suffices to prove the result in the case $\ell=1$.

For this, we first establish that $c(\xi_{std})$ is dual to $\theta_+$. This can be verified in a number of ways; for instance, through the calculation in Example \ref{Whitehead}.  $Wh^+$ is  a fibered knot in $\SoneStwo$, a fact implied by having rank one knot Floer homology in the top Alexander grading \cite{NiFibered} c.f. \cite{Ghiggini2007}. Figure \ref{fig:Agrading} shows the complex $\CFa(S^1\times S^2)$ with filtration induced by $Wh^+$.  The contact invariant of the contact structure associated to the open book coming from $Wh^+$ is, by definition, the element in the  homology of the dual  complex arising via inclusion of the bottommost non-trivial filtered subcomplex.  The dual complex and dual filtration are computed by reversing arrows and negating Alexander gradings, respectively, so in the case at hand $[d+f]^*=\theta_{+}^*\in \HFa^*(\SoneStwo)$ and $c(\xi_{Wh^+})=\theta_{+}^*$ since $\theta_{+}^*$ is in the bottommost filtration level in the dual complex.  Since $c(\xi_{Wh^+})\neq 0$, the contact structure $\xi_{Wh^+}$ is tight \cite[Theorem 1.4]{Contact}, and must therefore be isotopic to $\xi_{std}$, as the latter is the unique tight contact structure on $\SoneStwo$ \cite{Eliash1}. Thus, $c(\xi_{std})=c(\xi_{Wh^+})=\theta_+^*$. 

Now let $K\subset S^1\times S^2$ be any knot. If $\tau_{c(\xi_{std})}^*(K)=n$ then, by definition, there exists a class $\alpha\in\Image(I_n)$ such that $\langle c(\xi_{std}),\alpha\rangle \neq 0$. Moreover, since $\tau_{c(\xi_{std})}^*(K)$ is the minimum filtration index for which there exists such a class, it follows that $\tau_\alpha(K)=n$. Monotonicity now implies $\tau_\alpha(K)\leq \tau_{top}(K)$. On the other hand, since $$\langle c(\xi_{std}),\alpha\rangle=\langle \theta_+^*,\alpha \rangle \neq 0,$$ $\alpha$ decomposes as a sum $\theta_++\alpha'$ for some class $\alpha'$ pairing trivially with $\theta_+^*$. But any such class is a linear combination of classes in the image in the $H_1$ action which, by monotonicity,  are represented in filtration levels less than or equal to the minimum filtration level representing $\theta_+$.  By the subadditivity of $\tau$ (Proposition \ref{prop:sum}) we therefore have \[\tau_{top}(K):=\tau_{\theta_+}(K)=\tau_{\theta_++\alpha'+\alpha'}(K)\le \mathrm{max}\{\tau_{\alpha}(K),\tau_{\alpha'}(K)\}=\tau_{\alpha}(K),\] and hence $\tau_{top}(K)\leq \tau_{\alpha}(K)\le \tau_{top}(K)$.  We conclude $\tau_{top}(K)=\tau_\alpha(K)=n=\tau_{c(\xi_{std})}^*(K)$.
\end{proof}

The proof of Proposition \ref{prop:quasipositive}  relies on a Bennequin type inequality for links proved by the authors in \cite{4Dtight}, which we state here for the special case of links in $S^3$:

\begin{theorem}[$\tau$-Bennequin bound \cite{4Dtight}] Suppose $L\subset S^3$ is a link of $|L|$ components.  Then for any Legendrian representative $\mathcal{L}$ of $L$ in the standard tight contact structure on $S^3$ we have
\begin{equation} \label{eq:sBi}
\tb(\mathcal{L})+\rot(\mathcal{L})+|L|-1\leq 2\tau_{\xi_{std}}(\kappa({L}))-1.
\end{equation}
\end{theorem}
\noindent The calculation of $\tau_{top}$ for quasipositive links will now follow quickly.  The strategy, adapted from the case of knots from that in \cite{SQPfiber} was employed independently by Cavallo in \cite[Theorem 1.4]{CavalloBennequin}

\begin{proofof}{\bf Proposition \ref{prop:quasipositive}.} We have $\tau_{\xi_{std}}(\kappa({L}))=\tau_{top}(\kappa({L}))$ from Proposition \ref{prop:top/std}, and the latter is the definition of $\tau_{top}(L)$.  Substituting this in the $\tau$-Bennequin bound, and recalling the adjunction inequality, we obtain:
\[\tb(\mathcal{L})+\rot(\mathcal{L})+|L|-1\leq 2\tau_{top}(L)-1\le |L|-\chi(\Sigma)-1\]
 for any smoothly embedded surface $\Sigma$ with $\partial\Sigma=L$.  For quasipositive links, we demonstrate a Legendrian representative and surface $\Sigma$ for which the outer terms agree,  following the proof of \cite[Theorem 1.5]{SQPfiber}.

Let $\beta$ be a quasipositive braid representative for $L$. Let $n_+$ and $n_-$ denote the number of positive and negative generators, respectively, used in  the braid word and let $b$ be the braid index. Since $\beta$ is of the form $\prod_{k=1}^mw_k\sigma_{i_k} w_k^{-1}$, we have $n_+=n_-+m$. 

To obtain a Legendrian representative  for the closure of $\beta$, stabilize at each negative generator; see \cite[Figure 3]{SQPfiber}. Calculation for this Legendrian representative produces
$$\tb(\mathcal{L})=\{\text{writhe}\}-\#\{\text{left cusps}\}=n_+-2n_--b$$ and $$|\rot(\mathcal{L})|=|\#\{\text{down left cusps}\} - \#\{\text{up right cusps}\}|= n_-$$
Thus $\tb(\mathcal{L})+\rot(\mathcal{L})=n_+-n_--b=m-b$.

On the other hand, the expression of $L$ as the closure of a product of $m$ conjugates  of generators of the $b$ stranded braid group gives rise to a braided ribbon surface bounded by $L$ of Euler characteristic $\chi(\Sigma)=b-m$ \cite[Figure 2.5]{Rudolph1983-2}. It follows that $2\tau_{top}(L)-1= |L|-\chi(\Sigma)-1$, hence $\tau_{top}(L)=\frac{|L|-\chi(\Sigma)}{2}=g_4(L)$. Moreover, if $L$ is strongly quasipositive, then the ribbon surface for $L$ is a Seifert surface, so $\tautop(L)=g_4(L)=g_3(L)$.  
\end{proofof}

Proposition \ref{prop:quasipositive} implies Theorem \ref{thm:links}\ref{thm:milnor} stated in the introduction, by a result of Boileau and Orevkov \cite{BO2001} c.f. \cite{hayden2017quasipositive}.  Their result yields a converse to Rudolph's construction which, with  Rudolph's, equates the set of isotopy classes of links bounding complex curves in the round $4$-ball with the set arising as the closures of quasipositive braids.  From their work, one sees that the Euler characteristic of any complex curve bounded by $L$ is given by $b-m$ for any quasipositive representative.  Theorem \ref{thm:links}\ref{thm:milnor} follows at once.

In the special case of fibered knots, \cite[Theorem 1.2]{SQPfiber} provides a partial converse to Proposition \ref{prop:quasipositive}; namely, if a fibered knot satisfies $\tau(K)=g(K)$, then it is strongly quasipositive.  We extend this result to links, yielding Theorem \ref{thm:links}\ref{thm:links-SQP}:
\begin{theorem}\label{thm:SQP}
Suppose $L\subset S^3$ is fibered. $L$ is strongly quasipositive if and only if $$\tautop(L)=g_4(L)=g_3(L).$$
\end{theorem}

\noindent Before  proving this,  we recall some definitions and facts about fibered knots and contact structures. 

Denote by $(F,L)$ the open book decomposition induced by a fibered link $L\subset Y$ with fiber surface $F$. Such an open book decomposition induces a contact structure on the 3-manifold and we write $\xi_L$ for the contact structure induced by $(F,L)$.

\begin{lemma}\label{Lemma:contactsum}
If $L\subset S^3$ is a fibered link, then $\kappa(L)\subset \#^{|L|-1}\SoneStwo$ is also fibered and $\xi_{\kappa(L)}\simeq\xi_{L}\#\xi_{std}$ where $\xi_{std}$ is the unique tight contact structure on $ \#^{|L|-1}\SoneStwo$.
\end{lemma}

\begin{proof} Fix a fiber surface $F$ for $L$. Since $F$ is a fiber in a fibration of a connected 3-manifold (the link complement), it is necessarily connected.  Hence we can choose $|L|-1$ disjoint arcs embedded in $F$ with boundary in $\partial F$ so that the union of $\partial F$ with this collection of arcs is connected.  Fix points $p$ and $q$ in $S^2$ and let $B \subset \SoneStwo$ be the fibered 2-component link $S^1\times \{p\}\cup -S^1\times \{q\}$ with fiber surface an untwisted annulus. To a neighborhood of each arc in $F$, plumb a copy of the fiber surface for $B$. The result is an open book decomposition of $\#^{|L|-1}\SoneStwo$ where $\kappa(L)$ is a fibered knot whose fiber surface is $F$ with $|L|-1$ bands attached.  

It follows from \cite[Theorem 1.3]{Torisu} that the contact structure $\xi_{\kappa(L)}$ induced by the open book coming from $\kappa(L)$ is $\xi_L\# (\#^{|L|-1}\xi_B)$. It remains to show that $\xi_B$ is the unique tight contact structure on $S^1\times S^2$.   This is well-known, as the monodromy of the annular open book is the identity.  One can alternatively provide a Floer homological proof.

To this end, recall that the Giroux correspondence implies that the contact structure induced by an open book decomposition is unchanged by plumbing positive Hopf bands. Plumbing a single positive Hopf band to the fiber surface for $B$ yields a surface whose  boundary is the positive Whitehead knot, $Wh^+$.  The calculation of its knot Floer homology in  Example \ref{Whitehead} shows that $c(\xi_{Wh^+})\neq 0$ (as discussed in the proof of Proposition \ref{prop:top/std}),  and this implies that $\xi_{Wh^+}$ is tight. Since $S^1\times S^2$ supports a unique tight contact structure, we conclude that $\xi_{Wh^+}$, and therefore $\xi_{B}$, is isotopic to $\xi_{std}$.
\end{proof} 

Theorem \ref{thm:SQP} is now a consequence of the following proposition.
\begin{prop}
Let $L\subset S^3$ be a fibered link with fiber surface $F$. Then the following are equivalent:
\begin{enumerate}
\item $L$ is strongly quasipositive. 
\item The open book decomposition associated to $(F,L)$ induces the unique tight contact structure on $S^3$. 
\item $c(\xi_L)\neq 0$ where $c(\xi_{L})$ is the Ozsv\'ath-Szab\'o contact invariant of the contact structure induced by the  open book decomposition of $S^3$ associated to $(F,L)$.
\item $L$ satisfies $\tautop(L)=g_3(L).$
\end{enumerate}
\end{prop}
\begin{proof}  Our argument is similar to \cite[Proposition 2.1]{SQPfiber}.
Proposition \ref{prop:quasipositive}   showed  that $(1)\Rightarrow (4)$.  We now show that $(4)\Rightarrow (3)\Rightarrow (2)\Rightarrow (1)$.  
\medskip

\noindent $(4)\Rightarrow (3)$.  Assume $\tautop(L):=\tautop(\#^{|L|-1}\SoneStwo, \kappa(L))= g$.  By Proposition \ref{prop:top/std}, we therefore have $\tau^*_{c(\xi_{std})}(\#^{|L|-1}\SoneStwo, \kappa(L))=g$.   The duality Proposition \ref{prop:dualtau} implies
\[ \tau_{c(\xi_{std})}(-\#^{|L|-1}\SoneStwo, \kappa(L))=-\tau^*_{c(\xi_{std})}(\#^{|L|-1}\SoneStwo, \kappa(L))=-g,\]
which implies $c(\xi_{std})$ is  the image of $H_*(\Filt_{-g}(-\#^{|L|-1}\SoneStwo,\kappa(L)))\cong \F$ under the map induced on homology by the inclusion
$$\iota:\Filt_{-g}(-\#^{|L|-1}\SoneStwo,\kappa(L))\hookrightarrow \CFa(-\#^{|L|-1}\SoneStwo).$$
By the definition of the contact invariant, this means that the  invariant of the contact structure associated to the fibered knot $\kappa(L)$  equals  $c(\xi_{std})$.  But $\xi_{\kappa(L)}\simeq \xi_L\# \xi_{std}$ by Lemma \ref{Lemma:contactsum}, and the product formula for the contact invariant \cite[Property 4, pg. 105]{tbbounds} therefore yields 
$$c(\xi_{\kappa(L)})=c(\xi_L\#\xi_{std})=c(\xi_L)\otimes c(\xi_{std}).$$
Since $c(\xi_{\kappa(L)})=c(\xi_{std})\ne 0$, it follows that $c(\xi_L)\ne0\in \HFa(-S^3)$.

\medskip
\noindent $(3)\Rightarrow (2)$.  Non-vanishing of $c(\xi_L)$ implies tightness, which shows  the contact structure induced by $L$ is isotopic to the (unique) tight contact structure on $S^3$.

\medskip
\noindent $(2)\Rightarrow (1)$.  Since the unknot and $L$ both induce the tight contact structure on $S^3$, the Giroux correspondence implies $L$ is stably equivalent to the unknot; that is, the fiber surface for $L$ is obtained from a disk by plumbing and deplumbing positive Hopf bands. Rudolph showed, however, that a Murasugi sum of surfaces is quasipositive if and only if each of the summands is quasipositive \cite{Rudolph1998}. In particular, plumbing and deplumbing of positive Hopf bands preserves strong quasipositivity of the bounding links. Thus, $L$ must be strongly quasipositive.
\end{proof}

The first author and Kuzbary impose a group structure on link concordance classes by defining such a structure on concordance classes of knots in $\SoneStwo$, and showing that knotification descends to concordance \cite{Hedden-Kuzbary}.  It is natural to ask whether the image of links in $S^3$ under knotification generates the concordance group of knots in $\SoneStwo$, i.e. is every knot in $\SoneStwo$ (or its connected sums)  concordant to the knotification of a link?  The following  answers this question negatively.

\begin{prop}\label{prop:notknotified}  There are null-homologous knots in $\SoneStwo$ which are not concordant to the knotification of any link in $S^3$.

\end{prop}
\begin{proof}
If $\kappa(L)\subset \SoneStwo$, then $L$ must be a 2-component link and by Proposition \ref{prop:monotone2},
\begin{equation}\label{eqn:kappabound}
\tautop(\kappa(L))\leq \taubot(\kappa(L))+|L|-1=\taubot(\kappa(L))+1.
\end{equation} 
Consider a sequence of knots $K_n$ where $K_1= Wh^+$, and a diagram for $K_n$ is given by $n$ concentric copies of $Wh^+$ in $\SoneStwo$ joined by $n-1$ positive bands. Figure \ref{fig:knotinS1S2} shows the knot $K_3$. The knots $K_n$ are all null-homologous in $\SoneStwo$, and each $K_n$ bounds a disk in $D^2\times S^2$ analogous to the one constructed for $Wh^+$ in Example \ref{Whitehead}. Thus $\taubot(K_n)\le 0$. If $K_n$ is concordant to $\kappa(L)$, then $\taubot(K_n)=\taubot(\kappa(L))\leq 0$ and Equation \eqref{eqn:kappabound} gives $\tau_{top}(K_n)=\tautop(\kappa(L))\leq 1$. 

On the other hand, we can produce a Legendrian representative $\mathcal{K}_n$ for $K_n$ in the standard contact structure on $\SoneStwo$ satisfying  \[tb(\mathcal{K}_n)+rot(\mathcal{K}_n)=2n-1.\] Indeed, such a representative is obtained by trading  the 0-framed unknot in Figure \ref{fig:knotinS1S2} for a Stein 1-handle, over whom the $2n$ strands of $K_n$ pass, and  replacing the vertical tangencies of the $n$ positive clasps of $K_n$ with cusps.  The resulting  Legendrian knot in $(\SoneStwo,\xi_{std})$  has \[tb(\mathcal{K}_n)=\text{writhe}(\mathcal{K}_n)-\frac{\# \text{cusps}(\mathcal{K}_n)}{2}=(2n+n-1)-\frac{2n}{2}=2n-1\] and vanishing rotation number, according to the adaption of these invariants to the setting at hand, see \cite[Chapter 11]{GS}.

 Applying the main theorem of \cite{tbbounds}, we have $tb(\mathcal{K}_n)+rot(\mathcal{K}_n)\le 2\tau_{\xi_{std}}(K_n)-1$, so that $n\le \tau_{\xi_{std}}(K_n).$\footnote{In fact $n=\tau_{\xi_{std}}(K_n)$, since $\tau$ invariants are bounded by the Seifert genus, and there is a genus $n$ Seifert surface for $K_n$ by construction.}  But $\tau_{\xi_{std}}(K_n)=\tau_{top}(K_n)$, by Proposition \ref{prop:top/std}. Thus, for each $n>1$, $\tau_{top}(K_n)>1$ and $K_n$ is therefore not concordant to a knotified link.
\begin{figure}
\includegraphics[height=2in]{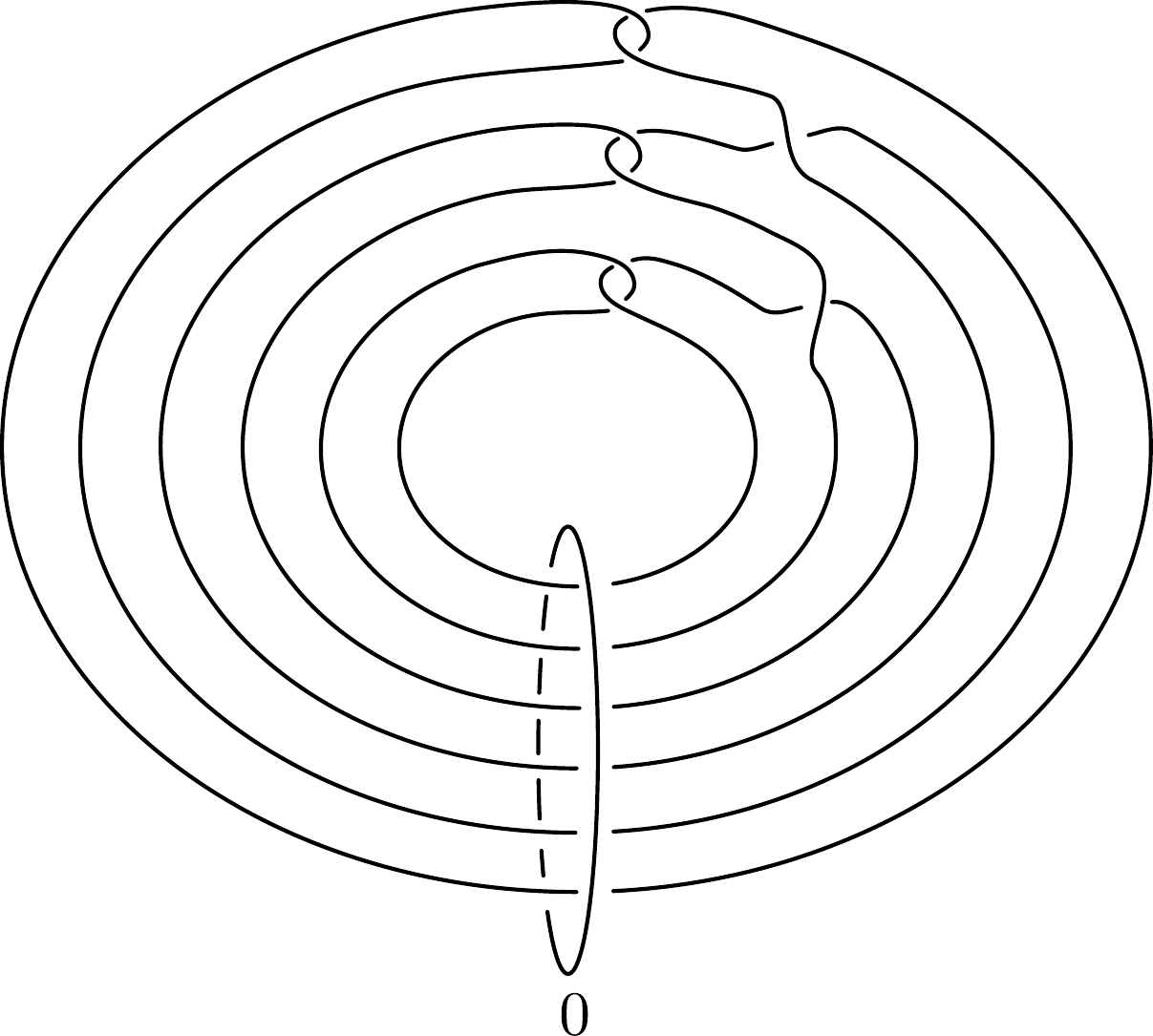}
\caption{The knot $K_3$ in $\SoneStwo$. Calculation shows  $\tautop(K_3)=3$.}
\label{fig:knotinS1S2}
\end{figure}
\end{proof}

\subsection{Comparison with other definitions of $\tau(L)$}\label{subsec:comparison}
In \cite{Links}, \ons\ define an invariant of a link $L\subset S^3$  taking the form of a graded, $\Z^{|L|}$ filtered complex whose graded Euler characteristic recovers the multivariable Alexander polynomial.  There are versions of this invariant for base ring  $\F$ or $\F[U]$, corresponding to the ``hat" and ``minus" versions of Floer homology, respectively.    Over either ring, the total homology of the complex has rank $2^{|L|-1}$ and, using this, one can derive numerical invariants in the spirit of our definition of $\tau_{\Theta}(L)$.  It is natural to ask how they compare.   In the special case of links in the 3-sphere, one can also define and compute the link Floer homology complex from a grid diagram \cite{MOS, MOST,GridBook}.  Within the context of grid homology, Cavallo defines a numerical invariant of links and shows that it satisfies many of the same properties of the $\tau$ invariants we define here.  In particular, he shows it is a concordance invariant \cite[Theorem 1.2]{Cavallo}, bounds the genera of surfaces bounded by the link in $B^4$ \cite[Proposition 1.4]{Cavallo}, detects strongly quasipositive fibered links \cite[Theorem 1.2]{CavalloSQP},  refines the slice-Bennequin inequality \cite[Proposition 1.5]{Cavallo} and is determined by the signature of $L$ for quasi-alternating links \cite[Proposition 1.1(iv)]{Cavallo}.    

\begin{theorem}\label{thm:tausequal}  The  invariant $\tautop(L)$ defined above is equal to the invariant $\tau(L)$ defined by Cavallo \cite{Cavallo} which, in turn, is equal to the invariant $\tau_{max}(L)$ defined by Ozsv{\'a}th-Szab{\'o}-Stipsicz \cite[Definition 8.3.3]{GridBook}.\end{theorem} 

\begin{proof}
Cavallo defines $\tau(L)$ as the (Alexander) filtration level that supports the highest (Maslov) graded subspace of the total homology of the ``simply blocked, bigraded grid complex of a knot" \cite[Definition 8.2.7]{GridBook}, while Ozsv{\'a}th-Szab{\'o}-Stipsicz define $\tau_{max}(L)$ as negative the minimal Alexander grading of any homogeneous element generating a free $\F[U]$ submodule of the ``collapsed grid homology".  Here, the ``simply blocked, bigraded grid complex of a knot" is $\Z$-filtered chain homotopy equivalent to the knot Floer homology  ``hat" complex $(\CFKa(L), \partial)$, and the ``collapsed grid homology"  is isomorphic as a  bigraded $\F[U]$-module to the  knot Floer ``minus" groups of  the link, $\HFKm(L)$.  These equivalences are immediate from the stabilization invariance of Floer homology with respect to index 0/3 Heegaard diagram stabilizations, together with the fact that a grid diagram is a suitably stabilized genus one Heegaard diagram adapted to a link, see \cite{MOS}.  In both instances, the Alexander grading on the knot Floer homology of the link is obtained by picking a Seifert surface compatible with the (implicit) orientation, and collapsing the Alexander multi-grading using this choice to a $\Z$-grading.  In \cite[Section 5]{Cavallo},  Cavallo shows that the $\tau$-sets coming from the hat complex and minus homology coincide, where the former records the filtration levels where the homology changes dimension and the latter   the negative of the Alexander gradings of generators of the free $\F[U]$ submodules of $\HFKm(L)$.  This has, as consequence, the equality $\tau(L)=\tau_{max}(L)$.  

It remains to identify our invariant, $\tautop(L),$ defined with respect to the filtration on the hat complex of the knotification of $L$, with the former.  For this, we appeal to an argument similar to the proof of \cite[Theorem 1.1]{Links}.  That theorem posits a graded isomorphism between the hat Floer homology groups of $\kappa(L)\subset \#^{|L|-1}S^1\times S^2$ and  the hat knot Floer homology groups of the link.  The proof of this isomorphism goes by way of a degeneration (and implicit gluing) argument for $J$-holomorphic curves in Lipshitz's cylindrical formulation of Heegaard Floer homology \cite{Lipshitz}.

More precisely, they consider a multi-pointed Heegaard diagram adapted to the link $L$, and from it derive a doubly pointed Heegaard diagram for the knotification $\kappa(L)$ by  surgering the initial Heegaard diagram along $|L|-1$ pairs of basepoints, each lying on different components of $L$.  They then consider a 2-parameter sequence of complex structures on this latter diagram, parametrized according to the neck length of the annulus glued in via the surgery and the placement of the basepoints along which the surgery is performed.  Taking independent Gromov limits with respect to the two parameters, they argue that for some complex structures on both diagrams, holomorphic curves missing both the basepoints in the diagram for $\kappa(L)$ must coincide with those missing all the basepoints on the diagram for $L$.

This argument does not extend, however, to give a correspondence between the holomorphic curves relevant to the collapsed Alexander multi-filtration for the multi-pointed diagram for $L$ and the Alexander filtration for $\kappa(L)$.  Indeed, the Gromov limit taken shows that a pseudo-holomorphic curve arising from the diagram for $\kappa(L)$, which drops the Alexander filtration by $k$, gives rise to a curve for the multi-pointed diagram for $L$ satisfying:
\[  \sum_{i=1}^{|L|}  (n_{z_i}(\phi)-  n_{w_i}(\phi)) = k\]
and the additional requirement that $n_{z_i}(\phi)=n_{w_{i+1}}(\phi)$ for all $i=1,...,|L|-1$.   While these curves are included in  the filtered boundary operator for $L$, the latter requirement is not present and hence the collapsed filtration for $L$ appears to count more curves.  

We deal with this using an idea suggested by Sucharit Sarkar, which is to instead surger the Heegaard diagram along pairs of $w$ basepoints for $L$, and simultaneously perform  0/3 stabilizations by adding an $\alpha/\beta$ pair of Hamiltonian isotopic curves running along each of the newly created necks.  We then place $w$ basepoints in the small bigons bounded by the new curves in the neck.  The result is a multi-pointed Heegaard diagram for $\kappa(L)$.   See \cite[Figure 2.2]{SarkarSigns}.    In this case, since we have not eliminated any basepoints, the definitions of the collapsed filtration for $L$ and the filtration for $\kappa(L)$ agree, and the Gromov limit and gluing arguments identify pseudo-holomorphic curves in the appropriate moduli spaces.  See \cite[Theorem 2.7]{SarkarSigns} for more details.  This shows that the $\Z$-filtered homotopy type of the hat complex for $\kappa(L)$ and the collapsed filtration of $L$ agree, after performing a sequence of 0/3-stabilizations to both.  But the effect of such a stabilization on the filtered homotopy type is to tensor both complexes with a filtered vector space of rank two, with summands differing in both filtration and homological grading  by one.   It follows that the filtration indices of the top graded summands of both complexes, and hence those of the complexes before stabilization, are equal.  But these indices are equivalent to the definitions of  $\tau_{top}(\kappa(L)):=\tau_{top}(L)$ and Cavallo's $\tau(L)$, respectively.  
\end{proof}

\begin{remark}  Since the definitions of  $\tau$ used by Cavallo and \os-Stipsicz take place in the context of grid diagrams, they do not admit a straightforward extension to yield the definite 4-manifold bound, Theorem \ref{thm:linkgenus}.  Moreover, an intrinsic proof of functoriality for grid homology has not appeared (though see \cite{GrahamThesis} for results in this direction).  This necessitates the usage of $\tau$ sets (or  Cavallo's $T$-function) in that context, rather than $\tau$ invariants associated to specific Floer classes $\Theta$.  We can define analogues of $\tau$ sets in our context, as the ordered collection of $\tau$ invariants associated to our basis elements $\theta_{\epsilon_1}\otimes \ldots \otimes\theta_{\epsilon_{|L|-1}}$, and the proof of Theorem \ref{thm:tausequal} should show that we will obtain the same set  as \os-Stipsicz.

\end{remark}

\begin{remark} The equivalence of invariants established in Theorem \ref{thm:tausequal} may be of computational significance for $\tau_{max}(L)$.  On its own this invariant is somewhat difficult to extract from grid homology. Cavallo's $\tau(L)$, on the other hand, can be derived solely from the subgroup in homological grading zero, and is therefore significantly simpler to compute. 
\end{remark}
\bibliographystyle{plain}
\bibliography{mybib}
\end{document}

%% file: macros.tex
\newcommand\commentable[1]{#1}

\newcommand\Id{\mathrm{Id}}

\newcommand{\HF}{HF}
\newtheorem{thm}{Theorem}

\newtheorem*{Adjunction Inequality}{Adjunction Inequality}
\newtheorem*{sympthom}{Symplectic Thom Conjecture}

\newtheorem{theorem}{Theorem}[section]
\newtheorem{prop}[theorem]{Proposition}
\newtheorem{cor}[theorem]{Corollary}

\newtheorem{lemma}[theorem]{Lemma}

\newtheorem{example}[theorem]{Example}
\newtheorem{defn}[theorem]{Definition}

\newtheorem{remark}[theorem]{Remark}

\def\endproofof{\relax\ifmmode\expandafter\endproofmath\else
  \unskip\nobreak\hfil\penalty50\hskip.75em\hbox{}\nobreak\hfil\bull
  {\parfillskip=0pt \finalhyphendemerits=0 \bigbreak}\fi}
\def\endproofofmath$${\eqno\bull$$\bigbreak}

\def\endproof{\relax\ifmmode\expandafter\endproofmath\else
  \unskip\nobreak\hfil\penalty50\hskip.75em\hbox{}\nobreak\hfil\bull
  {\parfillskip=0pt \finalhyphendemerits=0 \bigbreak}\fi}
\def\endproofmath$${\eqno\bull$$\bigbreak}
\def\bull{\vbox{\hrule\hbox{\vrule\kern3pt\vbox{\kern6pt}\kern3pt\vrule}\hrule}}

\newcommand{\Q}{\mathbb{Q}}
\newcommand{\R}{\mathbb{R}}

\newcommand{\C}{\mathbb{C}}

\newcommand{\Z}{\mathbb{Z}}

\newcommand{\CP}[1]{{\mathbb{CP}}^{#1}}
\newcommand{\CPbar}{{\overline{\mathbb{CP}}}^2}

\newcommand{\Image}{\mathrm{Im}}

\newcommand\SpinC{\mathrm{Spin}^c}
\newcommand{\F}{\mathbb F}

\newcommand\relspinc{\underline{\spinc}}

\newcommand\Filt{\mathcal F}

\newcommand\x{\mathbf x}

\newcommand\y{\mathbf y}

\newcommand\sC{\mathcal C}

\newcommand\ModSphere{\ModFlow\left({\mathbb S}\longrightarrow 
\Sym^{g-1}(\Sigma_{1})\times \Sym^2(\Sigma_{2})\right)}
\newcommand\ModSpheres\ModSphere
\newcommand\CF{CF}

\newcommand\CFa{\widehat{CF}}
\newcommand\CFp{\CFb}
\newcommand\CFm{\CF^-}

\newcommand\HFp{\HFb}

\newcommand\HFm{\HF^-}
\newcommand\CFinf{CF^\infty}
\newcommand\HFinf{HF^\infty}
\newcommand\CFb{CF^+}
\newcommand\HFa{\smash{\widehat{HF}}}
\newcommand\HFhat{\HFa}
\newcommand\HFb{HF^+}

\newcommand\UnparModSp{\widehat \ModSp}
\newcommand\UnparModFlow\UnparModSp
\newcommand\Mod\ModSp

\newcommand\PD{\mathrm{PD}}

\newcommand{\spinc}{\mathfrak s}
\newcommand{\spincr}{\mathfrak r}
\newcommand{\spincu}{\mathfrak u}
\newcommand{\spinct}{\mathfrak t}

\newcommand\ModMaps{\mathcal M}
\newcommand\ModSp\ModMaps

\newcommand\Ta{{\mathbb T}_{\alpha}}
\newcommand\Tb{{\mathbb T}_{\beta}}

\newcommand\alphas{\mbox{\boldmath$\alpha$}}

\newcommand\betas{\mbox{\boldmath$\beta$}}
\newcommand\gammas{\mbox{\boldmath$\gamma$}}

\newcommand\spincrel\relspinc

\newcommand\CFK{CFK}
\newcommand\HFK{HFK}

\newcommand\CFKa{\widehat\CFK}

\newcommand\CFKinf{\CFK^{\infty}}

\newcommand\HFKm{\HFK^-}

\newcommand\Dual{\mathcal D}
\newcommand\Duality\Dual

\newcommand\ons{Ozsv{\'a}th and Szab{\'o}}
\newcommand\os{{Ozsv{\'a}th-Szab{\'o}}}

\def\Hom{\operatorname{Hom}}

\newcommand\taua{\tau_\alpha(Y_1,K_1)}
\newcommand\taub{\tau_\beta(Y_2,K_2)}

\newcommand\SSurf{S}

\newcommand\tauad{\tau_{\alpha^{\otimes d}}(\#^dY_1,\#^dK_1)}
\newcommand\taubd{\tau_{\beta^{\otimes d}}(\#^dY_2,\#^dK_2)}

\newcommand\Sym{\mathrm{Sym}}

\newcommand\Surg{Y_{\framing}(K)}
\newcommand\Surgneg{Y_{\!-\framing}(K)}
\newcommand\Wsurg{W_{\framing}(K)}
\newcommand\Wsurgneg{W_{\!-\framing}(K)}
\newcommand\Wsurgdual{-W^{\dagger}_{\framing}(K)}

\newcommand\Whandles{\widehat{W}}

\newcommand\Wone{-W^{\dagger}_{\framing_1}(K_1) }
\newcommand\Wtwo{W_{\!-\framing_2}(K_2)}
\newcommand\Surgone{Y_{\framing_1}(K_1)}
\newcommand\Surgtwo{Y_{\!-\framing_2}(K_2)}
\newcommand\Maptwo{ F_{\!- \framing_2,\spinc_2,m_2}}
\newcommand\Mapone{ F^\dagger_{\framing_1,\spinc_1,m_1}}

\newcommand\Mapneg{ F_{\!- \framing,\spinc,m}}
\newcommand\Mappos{ F^\dagger_{\framing,\spinc, m}}

\newcommand\framing{n}
\newcommand\order{q}

\newcommand\Warcsum{W^{\otimes d}}
\newcommand\Sigmaarcsum{\Sigma^{\otimes d}}
\newcommand\SigmaSSd{\Sigma^{\otimes d}_{S_1,S_2}}

\newcommand\SigmaSS{\Sigma_{S_1,S_2}}

\newcommand\SigmaDD{\Sigma_{D_1, D_2}}
\newcommand\SSurfD{ D_{\SSurf}}
\newcommand\SSurfDone{D_{1, \SSurf_1}}
\newcommand\SSurfDtwo{D_{2,\SSurf_2}}
\newcommand\hatSigma{\Sigma}
\newcommand\SoneStwo{S^1\times S^2}

\newcommand\Cbundle{\partial\Dbundle}
\newcommand\Dbundle{\nu(\hatSigma)}
\newcommand{\bignatural}{\mathop{\vcenter{\hbox{\Large$\natural$}}}}

\newcommand\tautop{\tau_{top}}
\newcommand\taubot{\tau_{bot}}

\newcommand\can{\text{can}}

\newcommand\tb{\operatorname{tb}}
\newcommand\rot{\operatorname{rot}}

\newcommand\lcm{\operatorname{lcm}}

%% file: relativeadjunction.bbl
\begin{thebibliography}{10}

\bibitem{Baker-Etnyre}
Kenneth Baker and John Etnyre.
\newblock Rational linking and contact geometry.
\newblock In {\em Perspectives in analysis, geometry, and topology}, volume 296
  of {\em Progr. Math.}, pages 19--37. Birkh\"auser/Springer, New York, 2012.

\bibitem{BO2001}
Michel Boileau and Stephan~Y. Orevkov.
\newblock {Quasipositivit{\'e} d'une courbe analytique dans une boule
  pseudo-convexe}.
\newblock {\em C. R. Acad. Sci. Paris}, 332:825--830, 2001.

\bibitem{Calegari-Gordon}
Danny Calegari and Cameron Gordon.
\newblock Knots with small rational genus.
\newblock {\em Comment. Math. Helv.}, 88(1):85--130, 2013.

\bibitem{Cavallo}
Alberto Cavallo.
\newblock The concordance invariant tau in link grid homology.
\newblock {\em Algebr. Geom. Topol.}, 18(4):1917--1951, 2018.

\bibitem{CavalloSQP}
Alberto {Cavallo}.
\newblock {Detecting fibered strongly quasi-positive links}.
\newblock {\em arXiv e-prints}, page arXiv:2004.02233, April 2020.

\bibitem{CavalloBennequin}
Alberto Cavallo.
\newblock On {B}ennequin-type inequalities for links in tight contact
  3-manifolds.
\newblock {\em J. Knot Theory Ramifications}, 29(8):2050055, 20, 2020.

\bibitem{Eliash1}
Yakov Eliashberg.
\newblock Contact {$3$}-manifolds twenty years since {J}. {M}artinet's work.
\newblock {\em Ann. Inst. Fourier (Grenoble)}, 42(1-2):165--192, 1992.

\bibitem{Ghiggini2007}
Paolo Ghiggini.
\newblock {Knot Floer homology detects genus-one fibred knots}.
\newblock {\em Amer. J. Math.}, 130(5):1151--1169, 2008.

\bibitem{Goda}
Hiroshi Goda, Hiroshi Matsuda, and Takayuki Morifuji.
\newblock Knot {F}loer homology of {$(1,1)$}-knots.
\newblock {\em Geom. Dedicata}, 112:197--214, 2005.

\bibitem{Goldsmith}
Deborah~L. Goldsmith.
\newblock A linking invariant of classical link concordance.
\newblock In {\em Knot theory ({P}roc. {S}em., {P}lans-sur-{B}ex, 1977)},
  volume 685 of {\em Lecture Notes in Math.}, pages 135--170. Springer, Berlin,
  1978.

\bibitem{GS}
Robert~E. Gompf and Andr{\'a}s~I. Stipsicz.
\newblock {\em {$4$}-manifolds and {K}irby calculus}, volume~20 of {\em
  Graduate Studies in Mathematics}.
\newblock American Mathematical Society, Providence, RI, 1999.

\bibitem{GrahamThesis}
Matthew Graham.
\newblock {\em Studying {S}urfaces in 4-{D}imensional {S}pace {U}sing
  {C}ombinatorial {K}not {F}loer {H}omology}.
\newblock ProQuest LLC, Ann Arbor, MI, 2012.
\newblock Thesis (Ph.D.)--Brandeis University.

\bibitem{GRS}
J.~Elisenda Grigsby, Daniel Ruberman, and Sa{\v{s}}o Strle.
\newblock Knot concordance and {H}eegaard {F}loer homology invariants in
  branched covers.
\newblock {\em Geom. Topol.}, 12(4):2249--2275, 2008.

\bibitem{hayden2017quasipositive}
Kyle Hayden.
\newblock Quasipositive links and stein surfaces.
\newblock {\em arXiv e-prints}, page arXiv:1703.10150, 2017.

\bibitem{tbbounds}
Matthew Hedden.
\newblock An {O}zsv\'ath-{S}zab\'o {F}loer homology invariant of knots in a
  contact manifold.
\newblock {\em Adv. Math.}, 219(1):89--117, 2008.

\bibitem{SQPfiber}
Matthew Hedden.
\newblock Notions of positivity and the {O}zsv\'ath-{S}zab\'o concordance
  invariant.
\newblock {\em J. Knot Theory Ramifications}, 19(5):617--629, 2010.

\bibitem{Hedden-Kuzbary}
Matthew Hedden and Miriam Kuzbary.
\newblock {A link concordance group from knots in connected sums of $S^2 \times
  S^1$}.
\newblock {In preparation}.

\bibitem{Hedden-Levine-surgery}
Matthew {Hedden} and Adam~Simon {Levine}.
\newblock {A surgery formula for knot Floer homology}.
\newblock {\em arXiv e-prints}, page arXiv:1901.02488, January 2019.

\bibitem{STau}
Matthew Hedden and Philip Ording.
\newblock The {O}zsv\'ath-{S}zab\'o and {R}asmussen concordance invariants are
  not equal.
\newblock {\em Amer. J. Math.}, 130(2):441--453, 2008.

\bibitem{rationalcontact}
Matthew Hedden and Olga Plamenevskaya.
\newblock Dehn surgery, rational open books and knot {F}loer homology.
\newblock {\em Algebr. Geom. Topol.}, 13(3):1815--1856, 2013.

\bibitem{4Dtight}
Matthew Hedden and Katherine Raoux.
\newblock Four-dimensional aspects of tight contact 3-manifolds.
\newblock {\em Proceedings of the National Academy of Sciences}, 118(22), 2021.

\bibitem{HM}
Kristen Hendricks and Ciprian Manolescu.
\newblock Involutive {H}eegaard {F}loer homology.
\newblock {\em Duke Math. J.}, 166(7):1211--1299, 2017.

\bibitem{HLL}
Jennifer Hom, Adam~Simon Levine, and Tye Lidman.
\newblock Knot concordance in homology cobordisms.
\newblock {\em ArXiv e-prints}, page arXiv: 1801.07770, 2018.

\bibitem{2019arXiv190402735J}
Andr{\'a}s {Juh{\'a}sz}, Maggie {Miller}, and Ian {Zemke}.
\newblock {Knot cobordisms, bridge index, and torsion in Floer homology}.
\newblock {\em arXiv e-prints}, page arXiv:1904.02735, Apr 2019.

\bibitem{JTZ}
Andr{\'a}s {Juh{\'a}sz}, Dylan~P. {Thurston}, and Ian {Zemke}.
\newblock {Naturality and mapping class groups in Heegaard Floer homology}.
\newblock {\em arXiv e-prints}, page arXiv:1210.4996, October 2012.

\bibitem{2018arXiv181009158J}
Andr{\'a}s {Juh{\'a}sz} and Ian {Zemke}.
\newblock {Stabilization distance bounds from link Floer homology}.
\newblock {\em arXiv e-prints}, page arXiv:1810.09158, Oct 2018.

\bibitem{2018arXiv180409589J}
Andr\'{a}s Juh\'{a}sz and Ian Zemke.
\newblock Distinguishing slice disks using knot {F}loer homology.
\newblock {\em Selecta Math. (N.S.)}, 26(1):Paper No. 5, 18, 2020.

\bibitem{KM1}
Peter~B. Kronheimer and Tomasz~S. Mrowka.
\newblock Gauge theory for embedded surfaces. {I}.
\newblock {\em Topology}, 32(4):773--826, 1993.

\bibitem{KMThom}
Peter~B. Kronheimer and Tomasz~S. Mrowka.
\newblock The genus of embedded surfaces in the projective plane.
\newblock {\em Math. Res. Lett.}, 1(6):797--808, 1994.

\bibitem{KM2}
Peter~B. Kronheimer and Tomasz~S. Mrowka.
\newblock Gauge theory for embedded surfaces. {II}.
\newblock {\em Topology}, 34(1):37--97, 1995.

\bibitem{Lipshitz}
Robert Lipshitz.
\newblock A cylindrical reformulation of {H}eegaard {F}loer homology.
\newblock {\em Geom. Topol.}, 10:955--1097, 2006.

\bibitem{manolescu2019generalization}
Ciprian {Manolescu}, Marco {Marengon}, Sucharit {Sarkar}, and Michael {Willis}.
\newblock {A generalization of Rasmussen's invariant, with applications to
  surfaces in some four-manifolds}.
\newblock {\em arXiv e-prints}, page arXiv:1910.08195, October 2019.

\bibitem{MOS}
Ciprian Manolescu, Peter~S. Ozsv\'{a}th, and Sucharit Sarkar.
\newblock A combinatorial description of knot {F}loer homology.
\newblock {\em Ann. of Math. (2)}, 169(2):633--660, 2009.

\bibitem{MOST}
Ciprian Manolescu, Peter~S. Ozsv\'{a}th, Zolt\'{a}n Szab\'{o}, and Dylan~P.
  {Thurston}.
\newblock On combinatorial link {F}loer homology.
\newblock {\em Geom. Topol.}, 11:2339--2412, 2007.

\bibitem{2019arXiv190305772M}
Maggie {Miller} and Ian {Zemke}.
\newblock {Knot Floer homology and strongly homotopy-ribbon concordances}.
\newblock {\em arXiv e-prints}, page arXiv:1903.05772, Mar 2019.

\bibitem{MST}
John~W. Morgan, Zolt{\'a}n Szab{\'o}, and Clifford~Henry Taubes.
\newblock A product formula for the {S}eiberg-{W}itten invariants and the
  generalized {T}hom conjecture.
\newblock {\em J. Differential Geom.}, 44(4):706--788, 1996.

\bibitem{NiKnotLink}
Yi~Ni.
\newblock A note on knot {F}loer homology of links.
\newblock {\em Geom. Topol.}, 10:695--713, 2006.

\bibitem{NiFibered}
Yi~Ni.
\newblock Knot {F}loer homology detects fibred knots.
\newblock {\em Invent. Math.}, 170(3):577--608, 2007.

\bibitem{NiThurston}
Yi~Ni.
\newblock Link {F}loer homology detects the {T}hurston norm.
\newblock {\em Geom. Topol.}, 13(5):2991--3019, 2009.

\bibitem{Ni-Wu}
Yi~Ni and Zhongtao Wu.
\newblock Heegaard {F}loer correction terms and rational genus bounds.
\newblock {\em Adv. Math.}, 267:360--380, 2014.

\bibitem{GridBook}
Peter~S. Ozsv\'{a}th, Andr\'{a}s~I. Stipsicz, and Zolt\'{a}n Szab\'{o}.
\newblock {\em Grid homology for knots and links}, volume 208 of {\em
  Mathematical Surveys and Monographs}.
\newblock American Mathematical Society, Providence, RI, 2015.

\bibitem{HigherAdj}
Peter~S. Ozsv{\'a}th and Zolt{\'a}n Szab{\'o}.
\newblock Higher type adjunction inequalities in {S}eiberg-{W}itten theory.
\newblock {\em J. Differential Geom.}, 55(3):385--440, 2000.

\bibitem{SympThom}
Peter~S. Ozsv{\'a}th and Zolt{\'a}n Szab{\'o}.
\newblock The symplectic {T}hom conjecture.
\newblock {\em Ann. of Math. (2)}, 151(1):93--124, 2000.

\bibitem{AbsGrad}
Peter~S. Ozsv{\'a}th and Zolt{\'a}n Szab{\'o}.
\newblock Absolutely graded {F}loer homologies and intersection forms for
  four-manifolds with boundary.
\newblock {\em Adv. Math.}, 173(2):179--261, 2003.

\bibitem{AltKnots}
Peter~S. Ozsv{\'a}th and Zolt{\'a}n Szab{\'o}.
\newblock Heegaard {F}loer homology and alternating knots.
\newblock {\em Geom. Topol.}, 7:225--254, 2003.

\bibitem{FourBall}
Peter~S. Ozsv{\'a}th and Zolt{\'a}n Szab{\'o}.
\newblock Knot {F}loer homology and the four-ball genus.
\newblock {\em Geom. Topol.}, 7:615--639, 2003.

\bibitem{Knots}
Peter~S. Ozsv{\'a}th and Zolt{\'a}n Szab{\'o}.
\newblock Holomorphic disks and knot invariants.
\newblock {\em Adv. Math.}, 186(1):58--116, 2004.

\bibitem{HolDiskTwo}
Peter~S. Ozsv{\'a}th and Zolt{\'a}n Szab{\'o}.
\newblock Holomorphic disks and three-manifold invariants: properties and
  applications.
\newblock {\em Ann. of Math.}, 159(3):1159--1245, 2004.

\bibitem{HolDisk}
Peter~S. Ozsv{\'a}th and Zolt{\'a}n Szab{\'o}.
\newblock Holomorphic disks and topological invariants for closed
  three-manifolds.
\newblock {\em Ann. of Math. (2)}, 159(3):1027--1158, 2004.

\bibitem{Contact}
Peter~S. Ozsv{\'a}th and Zolt{\'a}n Szab{\'o}.
\newblock {Heegaard Floer homology and contact structures}.
\newblock {\em Duke Math. J.}, 129(1):39--61, 2005.

\bibitem{HolDiskFour}
Peter~S. Ozsv{\'a}th and Zolt{\'a}n Szab{\'o}.
\newblock Holomorphic triangles and invariants for smooth four--manifolds.
\newblock {\em Adv. Math.}, 202:326--400, 2006.

\bibitem{Links}
Peter~S. Ozsv{\'a}th and Zolt{\'a}n Szab{\'o}.
\newblock Holomorphic disks, link invariants and the multi-variable {A}lexander
  polynomial.
\newblock {\em Algebr. Geom. Topol.}, 8(2):615--692, 2008.

\bibitem{RationalSurgeries}
Peter~S. Ozsv\'{a}th and Zolt\'{a}n Szab\'{o}.
\newblock Knot {F}loer homology and rational surgeries.
\newblock {\em Algebr. Geom. Topol.}, 11(1):1--68, 2011.

\bibitem{Olga2004}
Olga Plamenevskaya.
\newblock {Bounds for the Thurston--Bennequin number from Floer homology}.
\newblock {\em Algebr. Geom. Topol.}, 11(4):547--561, 2004.

\bibitem{Raoux}
Katherine Raoux.
\newblock {$\tau$}--invariants for knots in rational homology spheres.
\newblock {\em Algebr. Geom. Topol.}, 20(4):1601--1640, 2020.

\bibitem{RasThesis}
Jacob Rasmussen.
\newblock {\em Floer homology and knot complements}.
\newblock PhD thesis, Harvard University, 2003.

\bibitem{Rudolph1983}
Lee Rudolph.
\newblock Algebraic functions and closed braids.
\newblock {\em Topology}, 22(2):191--202, 1983.

\bibitem{Rudolph1983-2}
Lee Rudolph.
\newblock Braided surfaces and {S}eifert ribbons for closed braids.
\newblock {\em Comment. Math. Helv.}, 58(1):1--37, 1983.

\bibitem{Rudolph1993}
Lee Rudolph.
\newblock Quasipositivity as an obstruction to sliceness.
\newblock {\em Bull. Amer. Math. Soc. (N.S.)}, 29(1):51--59, 1993.

\bibitem{Rudolph1998}
Lee Rudolph.
\newblock {Quasipositive plumbing (Constructions of quasipositive knots and
  links. V)}.
\newblock {\em Proc. Amer. Math. Soc.}, 126(1):257--267, 1998.

\bibitem{SarkarSigns}
Sucharit Sarkar.
\newblock A note on sign conventions in link {F}loer homology.
\newblock {\em Quantum Topol.}, 2(3):217--239, 2011.

\bibitem{Schneiderman}
Rob Schneiderman.
\newblock Algebraic linking numbers of knots in 3-manifolds.
\newblock {\em Algebr. Geom. Topol.}, 3:921--968, 2003.

\bibitem{Torisu}
Ichiro Torisu.
\newblock Convex contact structures and fibered links in 3-manifolds.
\newblock {\em Internat. Math. Res. Notices}, (9):441--454, 2000.

\bibitem{Wall}
C.~T.~C. Wall.
\newblock {\em Surgery on compact manifolds}, volume~69 of {\em Mathematical
  Surveys and Monographs}.
\newblock American Mathematical Society, Providence, RI, second edition, 1999.
\newblock Edited and with a foreword by A. A. Ranicki.

\bibitem{ZemkeGraphCob}
Ian {Zemke}.
\newblock {Graph cobordisms and Heegaard Floer homology}.
\newblock {\em arXiv e-prints}, page arXiv:1512.01184, December 2015.

\bibitem{2019arXiv190204050Z}
Ian Zemke.
\newblock Knot {F}loer homology obstructs ribbon concordance.
\newblock {\em Ann. of Math. (2)}, 190(3):931--947, 2019.

\bibitem{ZemkeGrading}
Ian Zemke.
\newblock Link cobordisms and absolute gradings on link {F}loer homology.
\newblock {\em Quantum Topol.}, 10(2):207--323, 2019.

\bibitem{ZemkeLink}
Ian Zemke.
\newblock Link cobordisms and functoriality in link {F}loer homology.
\newblock {\em J. Topol.}, 12(1):94--220, 2019.

\end{thebibliography}
